\documentclass[a4paper]{amsart}
\usepackage{amscd,amssymb,stmaryrd, enumerate}
\usepackage[mathcal]{euscript}
\usepackage[all]{xy}
\SelectTips{cm}{}

\usepackage{hyperref}

\renewcommand{\mod}{\operatorname{mod}\nolimits}
\newcommand{\Mod}{\operatorname{Mod}\nolimits}
\newcommand{\uMod}{\operatorname{\underline{Mod}}\nolimits}

\newcommand{\gr}{{\operatorname{gr}\nolimits}}

\newcommand{\Ass}{\operatorname{Ass}\nolimits}
\newcommand{\Hom}{\operatorname{Hom}\nolimits}

\newcommand{\uHom}{\operatorname{\underline{Hom}}\nolimits}
\renewcommand{\Im}{\operatorname{Im}\nolimits}
\newcommand{\Ker}{\operatorname{Ker}\nolimits}

\newcommand{\rad}{\operatorname{rad}\nolimits}
\newcommand{\rrad}{\mathfrak{j}}

\newcommand{\Ann}{\operatorname{Ann}\nolimits}

\newcommand{\Tot}{\operatorname{Tot}\nolimits}

\newcommand{\Ext}{\operatorname{Ext}\nolimits}
\newcommand{\chkExt}{\operatorname{\widehat{Ext}}\nolimits}

\newcommand{\op}{{\operatorname{op}\nolimits}}

\newcommand{\id}{{\operatorname{id}\nolimits}}

\newcommand{\mingen}{{\operatorname{mingen}\nolimits}}
\newcommand{\cx}{{\operatorname{cx}\nolimits}}
\newcommand{\comp}{\operatorname{\scriptstyle\circ}}
\newcommand{\fraka}{\mathfrak{a}}
\newcommand{\frakb}{\mathfrak{b}}
\newcommand{\frake}{\mathfrak{e}}
\newcommand{\frakg}{\mathfrak{g}}
\newcommand{\frakl}{\mathfrak{l}}
\newcommand{\frakm}{\mathfrak{m}}
\newcommand{\frakp}{\mathfrak{p}}
\newcommand{\frakq}{\mathfrak{q}}
\newcommand{\frakr}{\mathfrak{r}}

\newcommand{\End}{\operatorname{End}\nolimits}

\renewcommand{\L}{\Lambda}

\newcommand{\Z}{{\mathbb Z}}

\newcommand{\A}{{\mathcal A}}
\newcommand{\B}{{\mathcal B}}

\newcommand{\C}{{\mathcal C}}
\newcommand{\D}{{\mathcal D}}

\newcommand{\X}{{\mathcal X}}

\newcommand{\bfC}{{\mathbf C}}
\newcommand{\bfD}{{\mathbf D}}
\newcommand{\bfK}{{\mathbf K}}
\newcommand{\perf}{{\operatorname{perf}\nolimits}}

\newcommand{\extto}{\xrightarrow}

\newcommand{\Spec}{\operatorname{Spec}\nolimits}

\newcommand{\Supp}{\operatorname{Supp}\nolimits}
\newcommand{\Proj}{\operatorname{Proj}\nolimits}
\newcommand{\Thick}{\operatorname{Thick}\nolimits}
\newcommand{\MCM}{\operatorname{CM}\nolimits}
\newcommand{\uMCM}{\operatorname{\underline{CM}}\nolimits}
\newcommand{\HH}{\operatorname{HH}\nolimits}

\newcommand{\cone}{{\operatorname{cone}\nolimits}}

\newcommand{\kos}{/\!\!/} 

\theoremstyle{plain}
\newtheorem{lem}{Lemma}[section]
\newtheorem{prop}[lem]{Proposition}
\newtheorem{cor}[lem]{Corollary}
\newtheorem{thm}[lem]{Theorem}
\theoremstyle{definition}
\newtheorem{defin}[lem]{Definition}
\newtheorem{example}[lem]{Example}
\newtheorem{assumption}[lem]{Assumption}
\newtheorem*{ack}{Acknowledgement}
\theoremstyle{remark}
\newtheorem{remark}[lem]{Remark}

\usepackage{pstricks}
\usepackage{lscape}
\usepackage{comment}

\begin{document}

\title{Support varieties -- an axiomatic approach}
\author[Buan]{Aslak Bakke Buan}
\address{Aslak Bakke Buan and \O yvind Solberg\\
Institutt for matematiske fag\\
NTNU\\ N--7491 Trondheim\\ Norway}
\email{aslak.buan \textrm{and} oyvind.solberg \textrm{at} ntnu.no}
\author[Krause]{Henning Krause}
\address{Henning Krause\\
Fakult\"at f\"ur Mathematik\\
Universit\"at Bielefeld\\
D-33501 Bielefeld\\
Germany} 
\email{hkrause@math.uni-bielefeld.de}
\author[Snashall]{Nicole Snashall}
\address{Nicole Snashall\\ 
Department of Mathematics\\
University of Leicester\\
University Road\\
Leicester, LE1 7RH\\
England}
\email{njs5@le.ac.uk}
\author[Solberg]{\O yvind Solberg}

\date{\today}

\begin{abstract}
  We provide an axiomatic approach for studying support varieties of
  objects in a triangulated category via the action of a tensor
  triangulated category, where the tensor product is not necessarily
  symmetric. This is illustrated by examples, taken in particular from
  the representation theory of finite dimensional algebras.
\end{abstract}

\maketitle
\tableofcontents

\section*{Introduction}
The main purpose of this paper is to present a common framework where
most of the existing occurrences of support varieties fit in. Examples
of such are support varieties for finite dimensional algebras, finite
groups, restricted finite dimensional Lie algebras, smooth algebraic
groups, finite group schemes, stable homotopy categories and complete
intersections. This paper is an early thought of, but late arriving,
companion of \cite{BKS}.  Some of the results were presented in
\cite{S,Sol}. 

An inspiration for this work have been the notes on axiomatic stable
homotopy theory by Hovey, Palmieri and Strickland \cite{HPS1997},
where tensor triangulated categories play a central role. There is
also the more recent approach of Balmer towards a support theory for
tensor triangulated categories \cite{Ba2005, Ba2010} and closely
related a theory of support via central ring actions \cite{BIK1}.  The
purpose of this paper is to point out (i) that then one often misses a
vital underlying structure, namely a tensor triangulated category
acting on the category where the theory of support is constructed and
(ii) that one obtains a central ring action from the graded
endomorphism ring of the tensor identity of the acting tensor
triangulated category.  This point of view has been taken successfully
by Stevenson in \cite{S1,S2}, but there the tensor triangulated
category acting has a symmetric tensor product.  This is not
necessarily true in our setting as our prime example is to consider
the category of bimodules over a finite dimensional algebra $\L$,
which are projective both as a left and as a right $\L$-module.

The pivotal results for conceiving a theory of support varieties in a
non\-com\-mu\-ta\-tive setting were shown in \cite{Ev,Go,V} around
1960, where the group cohomology ring of a finite group is shown to be
Noetherian, and further structural results of the cohomology ring were
obtained in \cite{Q} in 1971. Then in 1981 J.\ F.\ Carlson defined a
theory of support varieties for finitely generated modules over a
group algebra of a finite group (see \cite{Ca,Ca2}). These papers
define the genesis of a theory of support varieties considered in
noncommutative settings, as they have served as a motivation providing
the means to associate geometric data to algebraic structures.  We
also have to adopt similar finiteness conditions to obtain a proper
theory of support varieties in our setting, following ideas in
\cite{EHSST,SS}. 

An interesting source of examples are the stable module categories of
finite dimensional Hopf algebras.  The Hopf structure gives rise to a
tensor product which is not necessarily symmetric. In some cases,
results from the theory of cocommutative Hopf algebras carry over to
the noncommutative setting \cite{PW2009,PW2015}, while other examples
exhibit some new phenomena \cite{BW2014}.

The rough outline of this paper is as follows:
\S\S\ref{section:1}--\ref{section:2} are devoted to the foundations of
triangulated categories with a tensor action. In
\S\S\ref{section:3}--\ref{section:6} the basic properties of support
varieties are discussed.  The final
\S\S\ref{section:7}--\ref{section:9} present various classes of
examples.

\begin{ack}
The authors thank the Mathematisches Forschungsinstitut Oberwolfach
for the unique opportunity and for the support of our two Research in
Pairs stays in 2002 and 2004, which made this work possible. Some of
this work was presented at the Mini-Workshop `Support Varieties' at
Oberwolfach in February 2009 (see \cite{Sol}).
\end{ack}

\section{Tensor categories and actions}\label{section:1}
A category with a tensor product is called a \emph{monoidal} or a
\emph{tensor category} in the literature. This section is devoted to
recalling the definition of a tensor category and an action of a
tensor category on another category (see \cite{JK,K}). The examples we
have in mind are mostly triangulated categories, in particular those
equipped with a suspension. Even though some of our results only
depend on having suspended categories, in the main results we are
assuming the presence of a triangulated structure.  Therefore we focus
throughout this paper on triangulated tensor categories and actions of
such on triangulated categories. We end the section by reviewing our
arsenal of examples of triangulated tensor categories and actions of
these.

Recall that an additive category $\C$ is a \emph{tensor category} if
$\C$ carries an additional structure
$(\C,\otimes,\frake,\fraka,\frakl,\frakr)$, where $-\otimes-\colon
\C\times\C \to \C$ is an additive bifunctor, $\frake$ is an object in
$\C$, and $\fraka \colon (-\otimes-)\otimes-\to -\otimes(-\otimes-)$
is an isomorphism of functors $\C\times\C\times\C \to \C$.
Furthermore, $\frakl \colon \frake\otimes-\to -$ and $\frakr\colon
-\otimes \frake \to -$ are isomorphisms of functors $\C\to \C$ making
the following diagrams commute for all objects $x$, $y$, $z$ and $w$
in $\C$: (Pentagon Axiom)
\[\xymatrix{
((x\otimes y)\otimes z)\otimes w
\ar[r]^{\fraka}\ar[d]^{\fraka\otimes 1} &
(x\otimes y)\otimes (z\otimes w) \ar[r]^{\fraka} &
x\otimes (y\otimes (z\otimes w)) \\
(x\otimes (y\otimes z))\otimes w \ar[rr]^{\fraka} & & 
x\otimes((y\otimes z)\otimes w) \ar[u]^{1\otimes \fraka} }\]
and (Triangle Axiom) 
\[\xymatrix{
(x\otimes \frake)\otimes y \ar[rr]^\fraka\ar[dr]_{\frakr\otimes 1} &
& 
x\otimes (\frake\otimes y)\ar[dl]^{1\otimes \frakl} \\
& x\otimes y & }\]
Recall that 
\begin{align}
\frakl_{\frake\otimes x} & = 1_\frake\otimes\frakl_x,\notag\\
\frakr_{x\otimes\frake} & = \frakr_x\otimes 1_\frake,\notag\\
\frakl_\frake & =\frakr_\frake \colon \frake\otimes\frake\to
\frake\notag
\end{align}
from (\cite[Lemma XI.2.3]{K}).  A \emph{suspended} category is a
category $\D$ equipped with an autoequivalence $T\colon \D\to \D$.

Now we recall the definition of a triangulated tensor category.  A
\emph{triangulated} tensor category (\cite{SA}) is a tensor category
$(\C,\otimes,\frake,\fraka,\frakl,\frakr)$ and at the same time a
triangulated category with a suspension $T\colon \C\to \C$, where
there exist isomorphisms of functors $\lambda\colon -\otimes T(-)\to
T(-\otimes-)$ and $\rho\colon T(-)\otimes-\to T(-\otimes-)$ from
$\C\times\C\to \C$ making the following diagrams commutative
\[\xymatrix{
\frake\otimes T(x) \ar[r]^\frakl\ar[d]^\lambda & T(x)\ar@{=}[d] \\
T(\frake\otimes x) \ar[r]^{T(\frakl)} & T(x)}\qquad
\xymatrix{
T(x)\otimes \frake \ar[d]^\rho\ar[r]^\frakr & T(x)\ar@{=}[d]\\
T(x\otimes \frake) \ar[r]^{T(\frakr)} & T(x)}\]
and the following diagram anti-commutative
\[\xymatrix{
T(x)\otimes T(y) \ar[r]^{\rho_{x,T(y)}} \ar[d]^{\lambda_{T(x),y}} &
T(x\otimes T(y))\ar[d]^{T(\lambda_{x,y})} \\
T(T(x)\otimes y)\ar[r]^{T(\rho_{x,y})} & T^2(x\otimes y)}\]
for all objects $x$ and $y$ in $\C$. 

By an action of a tensor category on a category we mean the
following.  Let $(\C,\otimes,\frake,\fraka,\frakl,\frakr)$ be a tensor
category, and let $\A$ be a category. An \emph{action} of $\C$ on $\A$
is defined by the following data (see \cite{JK}):
\begin{enumerate}[\rm(i)]
\item An additive bifunctor $-*-\colon \C\times \A\to \A$, 
\item A natural isomorphism $\alpha_{x,y,a}\colon (x\otimes y)*a
\to x*(y*a)$ for all $x$ and $y$ in $\C$ and $a$ in $\A$, 
\item A natural isomorphism $\frakl'_a\colon \frake* a\to a$
for all $a$ in $\A$, 
\end{enumerate}
where these satisfy the following commutative diagrams:
\[\xymatrix{
((x\otimes y)\otimes z)* a
\ar[r]^{\alpha}\ar[d]^{\fraka* 1} &
(x\otimes y)*(z* a) \ar[r]^{\alpha} &
x*(y*(z*a))\\
(x\otimes (y\otimes z))*a \ar[rr]^{\alpha} & & 
x*((y\otimes z)*a) \ar[u]^{1*\alpha}}\]

\[\xymatrix{
(\frake\otimes x)* a\ar[rr]^{\alpha_{\frake,x,a}}\ar[dr]_{\frakl * 1} &
& 
\frake*(x*a)\ar[dl]^{\frakl'_{x*a}} \\
& x * a & }\]
and
\[\xymatrix{
(x\otimes \frake)* a\ar[rr]^{\alpha_{x,\frake,a}}\ar[dr]_{\frakr *1} &
& 
x*(\frake *a)\ar[dl]^{1*\frakl'_a} \\
& x * a & }\]
for all $x$, $y$ and $z$ in $\C$ and $a$ in $\A$. Using that
$\frakl_\frake= \frakr_\frake$ one obtains immediately from the above
axioms that $\frakl'_{\frake*a}=1_\frake*\frakl'_a\colon
\frake*(\frake*a)\to \frake*a$ for all objects $a$ in $\A$. 

Finally we recall the definition of an action of a triangulated tensor
category on a triangulated category. Let
$(\C,\otimes,\frake,\fraka,\frakl,\frakr,T,\lambda,\rho)$ be a
triangulated tensor category, and let $\A=(\A,\Sigma)$ be a
triangulated category. Then we define an \emph{action} of $\C$ on $\A$
to be
\begin{enumerate}[\rm(i)]
\item a functor $-*-\colon \C\times \A\to \A$, a natural
isomorphism $\alpha_{x,y,a}$ and a natural isomorphism $\frakl'_a$ for
all $x$ and $y$ in $\C$ and $a$ in $\A$ as above, such that 
\item there exist isomorphisms $\lambda'$ and $\rho'$ between
  the functors 
\[\lambda'\colon -*\Sigma(-) \to \Sigma(-*-)\]
and 
\[\rho'\colon T(-)*- \to \Sigma(-*-)\]
when viewed as bifunctors from $\C\times\A$ to $\A$, and such that 
\item the diagram 
\[\xymatrix{
e*\Sigma(a) \ar[r]^{\frakl'_{\Sigma(a)}}\ar[d]^{\lambda'} & 
\Sigma(a)\ar@{=}[d]\\
\Sigma(e*a)\ar[r]^{\Sigma(\frakl'_a)} & \Sigma(a)}\]
commutes for all $a$ in $\A$, and such that 
\item there is an anti-commutative diagram 
\[\xymatrix{
T(x)*\Sigma(a) \ar[d]^{\lambda'_{T(x),a}}\ar[r]^{\rho'_{x,\Sigma(a)}} &
\Sigma(x*\Sigma(a)) \ar[d]^{\Sigma(\lambda'_{x,a})} \\
\Sigma(T(x)*a)\ar[r]^{\Sigma(\rho'_{x,a})} & \Sigma^2(x*a)}\]
for all $x$ in $\C$ and $a$ in $\A$. 
\end{enumerate}
\begin{remark}
Let $\C=(\C,\otimes,\frake,\fraka,\frakl,\frakr,T,\lambda,\rho)$ be a
triangulated tensor category. Then it follows directly from the
definition, that there is an action of $\C$ on $\C$ by letting
$-*-=-\otimes -$, $\alpha=\fraka$, $\frakl'=\frakl$,
$\lambda'=\lambda$ and $\rho'=\rho$. Also note that we do not assume
any exactness properties of the tensor product $-\otimes-$ in either
of the variables. It is only the graded structure through the shift in
the triangulated categories that is crucial for Section~\ref{section:2}.
\end{remark}

We end this section by giving some examples of triangulated tensor
categories with actions on triangulated categories. To do this it is
convenient to point out some elementary general facts about categories
of complexes.

Let $R$ be a ring. Denote by $\bfC(R)$ and $\bfC(\mod R)$ the category
of complexes of all left $R$-modules and all finitely presented left 
$R$-modules, respectively. The tensor product gives rise to a functor
$\bfC(R^\op)\times \bfC(R)\to \bfC(\Z)$ via the total complex. Our
conventions for the signs are the following. The shift of a complex
$X$ is given by $X[p]^n=X^{n-p}$ and $d_{X[p]}=(-1)^p d_{X}$ for any
integer $p$ in $\Z$. For a morphism $f\colon X\to Y$ of complexes
$f[p]^n=f^{n-p}$. Given a complex $X$ in $\bfC(R^\op)$ and a complex
$Y$ in $\bfC(R)$, the total complex $\Tot(X,Y)=X\otimes_R Y$ has
$(X\otimes_R Y)^n =\amalg_{i\in\Z} X^i\otimes_R Y^{n-i}$ with 
differential $d^n\colon (X\otimes_R Y)^n\to (X\otimes_R Y)^{n+1}$
given by $x^i\otimes y^{n-i}\mapsto d_X(x^i)\otimes y^{n-i} +
(-1)^ix^i\otimes d_Y(y^{n-i})$.

By abuse of notation let $R$ also denote the stalk complex with $R$
concentrated in degree zero. The multiplication maps $R\otimes_R M\to
M$ and $N\otimes_R R\to N$ for an $R$-module $M$ and an $R^\op$-module
$N$ induce natural isomorphisms $\frakr\colon X\otimes_R R\to X$ and
$\frakl\colon R\otimes_R Y\to Y$ for all complexes $X$ in
$\bfC(R^\op)$ and all complexes $Y$ in $\bfC(R)$. Hence $R$ is the
tensor identity in $\bfC(R)$.

Define $\lambda\colon X\otimes_R Y[-1]\to (X\otimes_R
Y)[-1]$ by letting 
\[\lambda^n=\amalg_{i\in\Z} (-1)^i \id_{X^i}\otimes
\id_{Y^{n-i+1}}\colon (X\otimes_R Y[-1])^n\to (X\otimes_R Y)[-1]^n.\]
Let $\rho\colon X[-1]\otimes_R Y\to (X\otimes_R Y)[-1]$ be given by 
\[\rho^n=\amalg_{i\in\Z}\id_{X^{i+1}}\otimes 
\id_{Y^{n-i}}\colon (X[-1]\otimes_R Y)^n\to (X\otimes_R Y)[-1]^n.\]
Both of these maps are isomorphisms, which are natural in each
variable. We leave it to the reader to check that the diagrams 
\[\xymatrix{
R\otimes_R X[-1] \ar[r]^\frakl\ar[d]^\lambda & X[-1]\ar@{=}[d] \\
(R\otimes_R X)[-1] \ar[r]^(.6){\frakl[-1]} & X[-1]}\qquad
\xymatrix{
X[-1]\otimes_R R \ar[d]^\rho\ar[r]^\frakr & X[-1]\ar@{=}[d]\\
(X\otimes_R R)[-1] \ar[r]^(.6){\frakr[-1]} & X[-1]}\]
are commutative 
and that the following diagram
\[\xymatrix@C=35pt{
X[-1]\otimes_R Y[-1] \ar[r]^(.48){\rho_{X,Y[-1]}} \ar[d]^{\lambda_{X[-1],Y}} &
(X\otimes_R Y[-1])[-1]\ar[d]^{\lambda_{X,Y}[-1]} \\
(X[-1]\otimes_R Y)[-1]\ar[r]^(.52){\rho_{X,Y}[-1]} & (X\otimes_R Y)[-2]}\]
is anti-commutative.

If we are in a setting where $(X\otimes_R Y)\otimes_S Z$ and
$X\otimes_R(Y\otimes_S Z)$ are defined, then there is an associativity
isomorphism between them induced by the associativity isomorphism for
tensor products of modules.  In addition the tensor product
$\bfC(S^\op)\times \bfC(S)\to \bfC(\Z)$ given by the total tensor
product over $S$ sends null homotopic maps to null homotopic maps,
such that the tensor product induces a functor $\bfK(S^\op)\times
\bfK(S)\to \bfK(\Z)$.  Here $\bfK(S)$ and $\bfK(\mod S)$ denote the
homotopy category of complexes of all left $S$-modules and all
finitely presented left $S$-modules over the ring $S$, respectively.
Having this in mind it is easy to check that the other requirements
for a triangulated tensor category are satisfied in the following
examples.
\begin{example}\label{exam:1}
  Let $R$ be a commutative ring. Then $\bfK(R)$ is a triangulated
  tensor category with the tensor product induced by the total tensor
  product over $R$, and with $\frake=R$, $\fraka$, $\frakl$, $\frakr$,
  $T$, and $\lambda$ and $\rho$ given as above. This gives rise to an
  action of $\bfK(R)$ on $\bfK(R)$.
\end{example}

\begin{example}\label{exam:2}
  Let $G$ be a finite group, and let $k$ be a field. Then $\bfK(kG)$
  is a triangulated tensor category with the tensor product induced by
  the total tensor product over $k$, and with $\frake=k$, $\fraka$,
  $\frakl$, $\frakr$, $T$, and $\lambda$ and $\rho$ given as
  above. Consequently there is an action of $\bfK(kG)$ on $\bfK(kG)$.
\end{example} 

\begin{example}\label{exam:3}
  An easy generalization of the above example is to consider a finite
  dimensional Hopf algebra $H$ over a field $k$. Then $\bfK(H)$ is a
  triangulated tensor category with the same choice of structures as
  for the group ring case. Hence there is an action of $\bfK(H)$ on
  $\bfK(H)$.
\end{example}

\begin{example}\label{exam:4}
  Let $\L$ be an algebra over a commutative ring $k$. Let
  $\L^e=\L\otimes_k \L^\op$ be the enveloping algebra of $\L$. Then
  $\bfK(\L^e)$ is a triangulated tensor category with the tensor
  product induced by the total tensor product over $\L$, and with
  $\frake=\L$, $\fraka$, $\frakl=\frakr$, $T$, and $\lambda$ and
  $\rho$ given as above. As above this gives rise to an action of
  $\bfK(\L^e)$ on $\bfK(\L^e)$. Furthermore, we obtain an action of
  $\bfK(\L^e)$ on $\bfK(\L)$ in a natural way.
\end{example}

In the examples $\bfK(kG)$ and $\bfK(H)$ the tensor product
$\otimes_k$ in $\Mod kG$ and $\Mod H$ is exact, so that the tensor
product of a complex with an acyclic complex is always an acyclic
complex again, or equivalently tensoring with a fixed complex
preserves quasi-isomorphisms. It follows from this that the tensor
product in the homotopy categories induces a tensor product on the
derived categories $\bfD(kG)$ and $\bfD(H)$. In addition, this induces
a triangulated tensor structure on $\bfD(kG)$ and $\bfD(H)$.

The situation is different for $\bfK(\L^e)$. Here, we restrict
to the full subcategory $\B$ in $\Mod\L^e$ (or $\mod\L^e$) consisting
of those $\L^e$-modules which are projective over $\L$ and
  $\L^\op$.  Then the tensor product $-\otimes_\L-$ is exact on
  $\B$. Let $\C=\bfD^{b}(\B)$ be the full subcategory of $\bfD(\L^e)$
  generated by all complexes of modules in $\B$ with bounded
  homology. Similarly as above, the tensor product $-\otimes_\L-$
  induces a tensor product on $\C$ making it a triangulated tensor
  category with the tensor structure induced from $\bfK(\L^e)$. This
  also gives rise to an action of $\C$ on $\bfD(\L)$, $\bfD^-(\mod\L)$
  and $\bfD^{b}(\mod\L)$.

The derived tensor product $-\otimes_\L^{\mathbb{L}} -$ on
$\bfD(\L^e)$ given by $X\otimes_\L p(Y)$ where $p(Y)\to Y$ is a
quasi-isomorphism and $p(Y)$ is a complex of projective modules, makes
$\bfD(\L^e)$ into a triangulated tensor category. Similarly, if $Y$ is
in $\bfD(\L)$, the derived tensor product $X\otimes_\L^{\mathbb{L}} Y$
yields an action of $\bfD(\L^e)$ on $\bfD(\L)$.

As above, for a commutative Noetherian ring $R$ the derived tensor
product $-\otimes_R^{\mathbb{L}}-$ on $\bfD(R)$ makes $\bfD(R)$ into a
triangulated tensor category. Consider the full subcategory
$\bfD^\perf(R)$ of perfect complexes and $\bfD^b(\mod R)$ under this
action. It is easy to see that the above action restricts to an action
of $\bfD^\perf(R)$ on $\bfD^b(\mod R)$, where the tensor product is
given by taking the total tensor product over $R$.

\begin{example}\label{exam:5}
  For a selfinjective algebra $\L$ let $\uMod\L$ denote the category
  $\Mod\L$ modulo the morphisms factoring through projective modules.
  This is a triangulated category with suspension given by the first
  negative syzygy, $\Omega^{-1}_\L$. In the stable categories $\uMod
  kG$ or $\uMod H$ there is an induced tensor product by the Hopf
  structure and since $P\otimes_k M$ and $M\otimes_k P$ are projective
  modules whenever $P$ is a projective module and $M$ is any
  module. In addition this tensor product induces exact functors
  (triangle functors) for a fixed object in each of the variables of
  the tensor product.  Hence we obtain that the stable categories
  $\uMod kG$ and $\uMod H$ are triangulated tensor categories.
\end{example}

\begin{example}\label{exam:6}
  Let $\L$ be a finite dimensional selfinjective algebra over a field
  $k$. Again let $\B$ denote the full subcategory of $\Mod\L^e$
  consisting of the bimodules projective as modules on either
  side. Since $P\otimes_\L B$ and $B\otimes_\L P$ are projective
  $\L^e$-modules whenever $P$ is a projective $\L^e$-module and $B$ is
  in $\B$, the tensor product $-\otimes_\L-$ also induces a tensor
  product on the stable category $\underline{\B}$ as a full
  subcategory of $\uMod\L^e$. As above the category $\underline{\B}$
  becomes a triangulated tensor category.
\end{example}

\begin{example}\label{exam:7}
  Let $\mathfrak S_d$ be the symmetric group permuting $d$ elements
  and let $k$ be a field. Let $n\ge d$ and set $V=k^n$. Then the Schur
  algebra $S_k(n,d)$ is by definition the endomorphism algebra
  $\End_{k\mathfrak S_d}(V^{\otimes d})$ and there exists an
  idempotent $e$ in $S_k(n,d)$ such that
  $e S_k(n,d)e\cong k\mathfrak S_d$. Multiplying with $e$ yields the
  \emph{Schur functor} $\mod S_k(n,d)\to\mod k\mathfrak S_d$ (see
  \cite{Gr1980}).

  The category $\mod S_k(n,d)$ carries a (not necessarily exact)
  symmetric tensor product \cite{Kr2013}. On the other hand,
  $\mod k\mathfrak S_d$ is a tensor category via $-\otimes_k -$ with
  the diagonal group action.  The Schur functor preserves the tensor
  product \cite{AR2014} and this yields an exact functor
  $\bfD^b(\mod S_k(n,d))\to\bfD^b(\mod k\mathfrak S_d)$ between
  triangulated tensor categories. In fact, it is a triangulated
  quotient functor \cite[Lemma 1.15]{Keller}. Thus the known
  classification of thick tensor ideals of
  $\bfD^b(\mod k\mathfrak S_d)$ via homogeneous prime ideals of the
  cohomology ring $H^*(\mathfrak S_d,k)$ (see \cite{BCR1997}) embeds
  into the presently unknown classification for
  $\bfD^b(\mod S_k(n,d))$.
\end{example}

\section{The endomorphism ring of the tensor identity}\label{section:2}
The endomorphism ring of the tensor identity in a suspended tensor
category was considered in \cite{SA} and shown to be
graded-commutative. Any homomorphism of graded rings from a positively
graded and graded-commutative ring $R$ to the graded centre of a
triangulated category is shown to give rise to a theory of support
varieties (see \cite{AI,BIK1,BIK2,BIKO}). This is called a
\emph{central ring action} of the graded ring $R$ on the triangulated
category $\A$.

This section is devoted to showing that there is a homomorphism of
graded rings from the graded endomorphism ring of the tensor identity
$\frake$ in a triangulated tensor category $\C$ to the graded centre
of a triangulated category $\A$ on which $\C$ is acting.  Hence it
gives rise to a central ring action on $\A$.

Let $\C=(\C,\otimes,\frake,\fraka,\frakl,\frakr,T,\lambda,\rho)$ be a
triangulated tensor category acting on a triangulated category
$\A=(\A,\Sigma)$.  Consider the graded endomorphism ring
$\End^*_\C(\frake)=\amalg_{p\in\Z}\Hom_\C(\frake, T^p(\frake))$ of the
tensor identity in $\C$, which clearly is a naturally $\Z$-graded ring
with multiplication given as follows: If $h\colon \frake\to
T^p(\frake)$ and $h'\colon \frake\to T^q(\frake)$, then
\[h\cdot h'=T^q(h)\comp h'\colon \frake\to T^{p+q}(\frake).\] Recall
that the \emph{graded centre} $Z^*(\A)$ of $\A$ is defined as the
graded ring which in degree $p$ in $\Z$ consists of all natural
transformations $z\colon \id_\A\to \Sigma^p$ such that
$\Sigma z=(-1)^p z\Sigma$ (see \cite{BF}). We want to define a
homomorphism of graded rings from $\End_\C^*(\frake)$ to $Z^*(\A)$. To
this end we need to study the induced isomorphisms
$x*\Sigma^p(a)\to \Sigma^p(x*a)$ and $T^p(x)*a\to \Sigma^p(x*a)$ for
all integers $p$.  Let $\lambda'_0$ and $\rho'_0$ be the identity
transformation of the functor $-*-\colon \C\times\A\to \A$. For $p>0$
let
\[\lambda'_p=\Sigma^{p-1}(\lambda')\comp\Sigma^{p-2}(\lambda')\comp\cdots
\comp\Sigma(\lambda')\comp \lambda'\colon -*\Sigma^p(-)\to \Sigma^p(-*-)\]
and
\[\rho'_p=\Sigma^{p-1}(\rho')\comp\Sigma^{p-2}(\rho')\comp\cdots
\comp\Sigma(\rho')\comp \rho'\colon T^p(-)*-\to \Sigma^p(-*-).\]
In particular, 
\[(\lambda'_p)^{-1}\colon \Sigma^p(-*\Sigma^{-p}(-))\to
-*\Sigma^p\Sigma^{-p}(-) \simeq -*-\]
when $(\lambda'_p)^{-1}$ is starting in $\Sigma^p(-*\Sigma^{-p}(-))$, 
and therefore
\[\Sigma^{-p}((\lambda'_p)^{-1})\colon -*\Sigma^{-p}(-)\to
\Sigma^{-p}(-*-)\] 
for $p>0$. Let $\lambda'_{-p}=\Sigma^{-p}((\lambda'_p)^{-1})$
for $p>0$. Similarly let 
\[\rho'_{-p}=\Sigma^{-p}((\rho'_p)^{-1})\colon T^{-p}(-)*-\to
\Sigma^{-p}(-*-)\] 
for $p>0$. With these definitions it is easy to check that
\begin{equation}\label{eq:leftidaction}
\Sigma^p(\frakl')\comp \lambda'_p =\frakl'\colon e*\Sigma^p(-)\to
\Sigma^p(-) 
\end{equation}
and 
\begin{equation}\label{eq:gradedcommactions}
\Sigma^q(\lambda'_p)\comp \rho'_q = (-1)^{pq} \Sigma^p(\rho_q')\comp
\lambda'_p \colon T^q(-)*\Sigma^p(-)\to \Sigma^{p+q}(-*-)
\end{equation}
for all integers $p$ and $q$. This last relation corresponds to the
diagram 
\[\xymatrix{
T^q(x)*\Sigma^p(a) \ar[d]^{\lambda'_p}\ar[r]^{\rho'_q} &
\Sigma^q(x*\Sigma^p(a)) \ar[d]^{\Sigma^q(\lambda'_p)} \\
\Sigma^p(T^q(x)*a)\ar[r]^{\Sigma^p(\rho'_q)} & \Sigma^{p+q}(x*a)}\]
being commutative up to the sign $(-1)^{pq}$. 

Let $h\colon \frake\to T^p(\frake)$ be a degree $p$ element in
$\End_\C^*(\frake)$. Then consider the following composition of
natural transformations of functors  
\[\id_\A\extto{\frakl'^{-1}} \frake*-\extto{h*1} T^p(\frake)*-
\extto{\rho_p'} \Sigma^p(\frake*-)\extto{\Sigma^p(\frakl')}
\Sigma^p(-),\] 
which we denote by $\varphi_\A(h)$.  We show that $\varphi_\A$
gives rise to a homomorphism of graded rings
$\varphi_\A\colon \End_\C^*(\frake)\to Z^*(\A)$.
\begin{prop}
The map $\varphi_\A\colon \End_\C^*(\frake)\to Z^*(\A)$ is a homomorphism
of graded rings. 
\end{prop}
\begin{proof}
We need to show that $\Sigma\varphi_\A(h) = (-1)^p\varphi_\A(h)\Sigma$ for
$h\colon \frake\to T^p(\frake)$. Consider the following diagram 
\[\xymatrix@C=35pt{
\Sigma \ar[r]^-{\Sigma(\frakl'^{-1})}\ar@{=}[d] & 
\Sigma(\frake*-)\ar[r]^-{\Sigma(h*1)}\ar[d]^-{\lambda'^{-1}} & 
\Sigma(T^p(\frake)*-)\ar[r]^-{\Sigma(\rho_p')}\ar[d]^{\lambda'^{-1}} & 
\Sigma(\Sigma^p(\frake*-))\ar[r]^-{\Sigma^{p+1}(\frakl')}
\ar[d]^-{\Sigma^p(\lambda'^{-1})} & 
\Sigma\Sigma^p(-)\\
\Sigma\ar[r]^-{\frakl'^{-1}} & 
\frake*\Sigma(-)\ar[r]^-{h*1} & 
T^p(\frake)*\Sigma(-)\ar[r]^-{\rho_p'} & 
\Sigma^p(\frake*\Sigma(-))\ar[r]^-{\Sigma(\frakl')} & 
\Sigma^p\Sigma(-)\ar@{=}[u]
}\]
The leftmost and the rightmost squares commute due to
\eqref{eq:leftidaction}. The second square commutes since $\lambda'$
is a morphism of functors. The third square commutes up to the sign
$(-1)^p$ by \eqref{eq:gradedcommactions}. Hence it follows that
$\varphi_\A(h)$ is in $Z^*(\A)$.  It is straightforward to check that
$\varphi_\A$ is a homomorphism of graded rings.
\end{proof}
Let 
\[\Hom_\A^*(a,b)=\amalg_{p\in\Z}\Hom_\A(a,\Sigma^p(b))\] 
for any objects $a$ and $b$ in $\A$, and let
$\End_\A^*(a)=\Hom_\A^*(a,a)$. The homomorphism set $\Hom_\A^*(a,b)$
is endowed with a left and a right module structure from
$\End_\A^*(b)$ and $\End_\A^*(a)$, respectively. For each object $a$
in $\A$ the evaluation at $a$ induces a homomorphism of graded rings
$\gamma_a\colon Z^*(\A)\to \End_\A^*(a)$ given by
$\gamma_a(\eta)=\eta_a\colon a\to \Sigma^p(a)$ for $\eta\colon
\id_\A\to \Sigma^p$ in $Z^*(\A)$. Then $\Hom_\A^*(a,b)$ has a left and
a right $Z^*(\A)$-module structure via the ring homomorphisms 
$\gamma_b$ and $\gamma_a$ respectively. For completeness we recall the
following.

\begin{prop}\label{prop:gradcomm}
The action of $Z^*(\A)$ on the right and on the left of
$\Hom^*_\A(a,b)$ for $a$ and $b$ in $\A$ satisfies, for $\eta\colon
\id_\A\to \Sigma^p$ in $Z^*(\A)$ and $f\colon a\to \Sigma^q b$ in
$\Hom^*_\A(a,b)$, the following equality
\[ \eta\cdot f = (-1)^{pq}f\cdot \eta.\]
\end{prop}
\begin{proof}
Let $\eta\colon \id_\A\to \Sigma^p$ be in $Z^*(\A)$ and $f\colon a\to
\Sigma^q(b)$ in $\Hom_\A^*(a,b)$. Since $\eta$ is a natural
transformation of functors, the following diagram commutes 
\[\xymatrix{
a\ar[r]^f\ar[d]^{\eta_a} & \Sigma^q(b)\ar[d]^{\eta_{\Sigma^q(b)}} \\
\Sigma^p(a)\ar[r]^{\Sigma^p(f)} & \Sigma^p\Sigma^q(b)
}\]
As $\eta_{\Sigma^q(b)}=(-1)^{pq}\Sigma^q(\eta_b)$, the claim follows. 
\end{proof}
Using that a tensor triangulated category $\C$ acts on itself, we
obtain the following immediate corollary.
\begin{cor}
\begin{enumerate}[\rm(a)]
\item The composition 
\[\End_\C^*(\frake)\extto{\varphi_\C}
Z^*(\C)\extto{\gamma_\frake} \End_\C^*(\frake)\] 
of homomorphisms of graded rings is the identity. 
\item  The graded endomorphism ring $\End_\C^*(\frake)$ is
  graded-commutative.  
\end{enumerate}
\end{cor}
\begin{proof}
The proof of (a) is a direct computation. The claim in (b) is then an
immediate consequence of Proposition~\ref{prop:gradcomm}. 
\end{proof}

\begin{remark} (1) The  triangulated  tensor category $\C$ with an
action on $\A$ can be viewed as a categorification of a central ring
action, namely, a homomorphism of graded rings from a
graded-commutative ring $R$ to $Z^*(\A)$. 

(2) The above gives rise to a homomorphism of graded rings
$\varphi_a\colon
\End_\C^*(\frake)\to \End_\A^*(a)$ for any object $a$ in $\A$ by
letting $\varphi_a=\gamma_a\varphi_\A$. For $h\colon \frake\to 
T^p(\frake)$, the morphism $\varphi_a(h)$ is given as
\[a\extto{\frakl'^{-1}} \frake*a\extto{h*1}
T^p(\frake)*a\extto{\rho_p'}
\Sigma^p(\frake*a)\extto{\Sigma^p(\frakl')} \Sigma^p(a).\]

(3) Suppose that idempotents split in $\C$ and in $\A$. Then, if
$\End_\C^0(\frake)$ decomposes as a ring, then the categories $\C$ and
$\A$ also decompose as categories. Hence we can always assume that
$\C$ and $\A$ are indecomposable as categories and therefore that
$\End_\C^*(\frake)$ is indecomposable as a ring.

(4) The statement in (b) was first shown in \cite[Theorem 1.7]{SA}.
As pointed out in that paper, we obtain the graded-commutativity of
the following graded rings (using the notation of Section~\ref{section:1}
and Examples~\ref{exam:1}--\ref{exam:6}):
\begin{enumerate}[\rm(i)]
\item Let $G$ be a finite group, and let $k$ be a field. Then
\[\End^*_{\bfD(kG)}(k)=\amalg_{p\in\Z}\Hom_{\bfD(kG)}(k,k[p])
\simeq \amalg_{p\geqslant 0} \Ext^p_{kG}(k,k),\]
is the group cohomology ring of $G$. Also
\[\End^*_{\uMod kG}(k)=\amalg_{p\in\Z}\uHom_{kG}(k,\Omega_{kG}^{-p}(k))
\simeq \chkExt^*_{kG}(k,k),\]
is the Tate cohomology ring of $G$. 
\item Let $H$ be a Hopf algebra over a field $k$. Then 
\[\End^*_{\bfD(H)}(k)=\amalg_{p\in\Z}\Hom_{\bfD(H)}(k,k[p])
\simeq \amalg_{p\geqslant 0} \Ext^p_H(k,k),\]
is the cohomology ring of $k$ over $H$. Also 
\[\End^*_{\uMod H}(k)=\amalg_{p\in\Z}\uHom_H(k,\Omega_H^{-p}(k))
\simeq \chkExt^*_H(k,k),\]
is the Tate cohomology ring of $k$ over $H$. 
\item Let $\L$ be an algebra over a field $k$. Then
\[\End^*_{\bfD(\L^e)}(\L)=\amalg_{p\in\Z}\Hom_{\bfD(\L^e)}(\L,\L[p])
\simeq \amalg_{p\geqslant 0} \Ext^p_{\L^e}(\L,\L),\]
is the Hochschild cohomology ring of $\L$ over $k$. Also, if $\L$ is
selfinjective,  
\[\End^*_{\underline{\B}}(\L)=\amalg_{p\in\Z}\uHom_{\L^e}(\L,
\Omega_{\L^e}^{-p}(\L))
\simeq \chkExt^*_{\L^e}(\L,\L),\]
is the Tate cohomology ring of $\L$ over $\L^e$.
\end{enumerate}
\end{remark}

\section{Support varieties}\label{section:3}

Throughout this section
$\C=(\C,\otimes,\frake,\fraka,\frakl,\frakr,T,\lambda,\rho)$ is a
triangulated tensor category acting on a small triangulated category
$\A=(\A,\Sigma)$. Let $H$ be a positively graded and
graded-commutative ring with a homomorphism of graded rings
$H\to \End_\C^*(\frake)$. As mentioned earlier this gives rise to a
theory of support varieties in $\Spec H$, where $\Spec H$ is the set
of all homogeneous prime ideals in $H$.  We begin this section by
pointing out the standard properties of these support varieties. We
then give realizability results for closed homogeneous subvarieties of
varieties of given objects and possible generators for $\A$.

In order to obtain our results, further assumptions are needed. So the
following are our standing assumptions.
\begin{assumption}\label{ass:standing}
For $\C$, $\A$ and $H$ the following holds:
\begin{enumerate}[(1)]
\item $\C=(\C,\otimes,\frake,\fraka,\frakl,\frakr,T,\lambda,\rho)$ 
is a triangulated tensor category acting on a small triangulated category
$\A=(\A,\Sigma)$. 
\item $H$ is a positively graded-commutative Noetherian ring with
  a homomorphism of graded rings $H\to \End_\C^*(\frake)$.
\item The left $H$-module $\Hom_\A^*(a,b)$ is finitely generated
  for all objects $a$, $b$ in $\A$. 
\end{enumerate}
\end{assumption}
If the graded ring $H$ has a non-trivial idempotent $f$ in degree
zero, and for some object $a$ in $\A$ we have both $f*1_a$ and
$(1_H - f)*1_a$ non-zero, then assuming that idempotents split in
$\A$, one can show that the category $\A$ decomposes. Hence, in this 
case, we can assume that the graded ring $H$ has only trivial
idempotents in degree zero.  We sometimes assume a stronger condition,
namely that $H^0$ is a local ring.

In general the graded endomorphism ring $\End_\C^*(\frake)$ need not
be a positively graded ring making $\Spec \End_\C^*(\frake)$ a more
difficult object to handle than $\Spec H$. One could use the positive
part $\End_\C^{\geqslant 0}(\frake)$ of $\End_\C^*(\frake)$ instead of
some graded-commutative ring $H$. However, there are situations where
assuming finite generation over $H$ or over
$\End_\C^{\geqslant 0}(\frake)$, are equivalent, which we now demonstrate. 

Having $\C$ acting on $\A$ gives rise to a functor from $\C$ to the
endofunctors of $\A$. A necessary condition related for this functor
to have a right adjoint, is as pointed out in \cite{JK}, that each
functor $-*a\colon \C\to \A$ has a right adjoint, that is, there is a
functor $\A\to\C$ for each object $a$ in $\A$, denoted $F''(a,-)$ and
an isomorphism
\[\Hom_\A(x*a,b)\to \Hom_\C(x,F''(a,b)),\]
natural in all three variables. Having such a right adjoint induces an
isomorphism $\Hom_\A^*(a,b)\simeq \Hom_\C^*(\frake,F''(a,b))$ of
$\End_\C^*(\frake)$-modules.  In some situations there are objects $a$
and $b$ in $\A$ such that $\frake$ is in the thick triangulated
subcategory generated by $F''(a,b)$ in $\C$. Therefore, if
$\Hom_\A^*(a,b)$ is finitely generated for some positively graded
commutative Noetherian ring $H$, then $\End_\C^{\geqslant 0}(\frake)$
(also $\End_\C^*(\frake)$) is Noetherian too. A further discussion on
such functors $F''$, called function objects, can be found in Section
\ref{section:5}. A classical isomorphism, which gives rise to such a
function object, is the adjunction isomorphism
\[\Hom_\L(B\otimes_\L M,N)\simeq \Hom_{\L^e}(B,\Hom_k(M,N))\]
for a $k$-algebra $\L$, where $B$ is a $\L^e$-module, and $M$ and $N$ 
are $\L$-modules.

Now we give the definition of the support variety of a pair of objects
$(a,b)$ in $\A$. 
\begin{defin}
For a pair of objects $a$ and $b$ in $\A$, the \emph{support variety
$V(a,b)$ of $(a,b)$ with respect to $H$} is given by
\[V(a,b)=\{\frakp\in \Spec H\mid \Hom_\A^*(a,b)_\frakp\neq (0)\} =
         \Supp(\Hom_\A^*(a,b)).\]
\end{defin}
Proposition~\ref{prop:gradcomm} implies that the annihilator
$\Ann_H\Hom_\A^*(a,b)$ of $\Hom_\A^*(a,b)$ as an $H$-module for any
objects $a$ and $b$ in $\A$, is independent of viewing
$\Hom_\A^*(a,b)$ as a left or as a right $H$-module. We denote this
annihilator by $A(a,b)$. For a graded ideal $I$ in $H$ we denote by
$V(I)=\Supp(H/I)$. 

The following properties of the support variety are standard and
straightforward to verify, and we leave the proofs to the reader.
\begin{prop}\label{prop:elementarypropvar}
\pushQED{\qed} The support variety $V(-,-)$ has the following properties:
\begin{enumerate}[\rm(a)]
\item Let $a_1\to a_2\to a_3\to \Sigma(a_1)$ be a triangle in
  $\A$. Let $a$ be an object in $\A$.
\begin{enumerate}[\rm(i)]
\item $V(a,a_r) \subseteq V(a,a_s)\cup
  V(a,a_t)$ whenever $\{r,s,t\}=\{1,2,3\}$. 
\item $V(a_r,a) \subseteq V(a_s,a)\cup
  V(a_t,a)$ whenever $\{r,s,t\}=\{1,2,3\}$. 
\end{enumerate}
\item $V(a,b)=V(\Sigma^i(a),\Sigma^j(b))$ for any pair of objects
  $(a,b)$ in $\A$ and integers $i$ and $j$ in $\Z$.
\item Let $\{a_i\}_{i=1}^r$ and $\{b_j\}_{j=1}^s$ be two finite
  sets of objects in $\A$. Then
\[V(\amalg_{i=1}^r a_i,\amalg_{j=1}^s b_j)=\cup_{i,j=1}^{r,s}
V(a_i,b_j).\qedhere\]
\end{enumerate} 
\end{prop}
Since the action of $H$ on $\Hom_\A^*(a,b)$ factors through the action
of $H$ on both $\Hom_\A^*(a,a)$ and $\Hom_\A^*(b,b)$ for any pair of
objects $a$ and $b$ in $\A$, the following result is immediate. 
\begin{prop}\label{prop:annihilator}
Let $a$ and $b$ be objects in $\A$. 
\begin{enumerate}[\rm(a)]
\item $V(a,b) \subseteq V(a,a)\cap V(b,b)$. 
\item $V(a,a) = \cup_{x\in \A}V(a,x) = \cup_{x\in\A} V(x,a)$. \qed
\end{enumerate} 
\end{prop}
Having these properties at hand we define the \emph{support variety}
of an object $a$ in $\A$ to be $V(a)=V(a,a)$. The properties above give
the following behaviour.
\begin{prop}\label{prop:elempropvariety}
 The support variety $V(-)$ has the following properties:
\begin{enumerate}[\rm(a)]
\item If $a_1\to a_2\to a_3\to \Sigma(a_1)$ is an exact triangle
  in $\A$, then $V(a_r)\subseteq V(a_s)\cup V(a_t)$
  whenever $\{r,s,t\}=\{1,2,3\}$. 
\item $V(a)=V(\Sigma^i(a))$ for all objects $a$ in $\A$ and $i$
  in $\Z$. 
\item $V(\amalg_{i=1}^n a_i)=\cup_{i=1}^n V(a_i)$.\qed
\end{enumerate}
\end{prop}

Our next aim is to show that any closed homogeneous subvariety of the
variety of an object $a$ in $\A$ occurs as a variety of an object in
$\A$. In doing so the construction of Koszul objects is crucial (see
\cite[\S6]{HPS1997}). Any morphism $h\colon \frake\to T^p(\frake)$
induces for any object $a$ in $\A$ a morphism
\[h* 1_a\colon \frake*a \to T^p(\frake)*a\]
which we can identify with 
\[h\cdot 1_a\colon a\simeq \frake*a\extto{h*1_a} T^p(\frake)*a\simeq
\Sigma^p(\frake*a)\simeq \Sigma^p(a).\]
Complete this morphism to a triangle
\[
a \extto{h\cdot 1_a} \Sigma^p(a)\to a\kos h\to\Sigma(a)
\]
in $\A$. An immediate consequence of the above construction is that
the Koszul object $a\kos h$ is in the thick subcategory generated by
$a$ in $\A$ for all homogeneous elements $h$ in $H$. Moreover, as we
also note below, $V(a\kos h)\subseteq V(a)$.

Using the triangle $a\extto{h\cdot 1_a} \Sigma^p(a) \to a\kos h \to
\Sigma(a)$ we have the following.

\begin{prop}\label{prop:decompannihilator}
  Let $h\colon \frake\to T^p(\frake)$ be in $\C$. Then the following
  assertions hold.
\begin{enumerate}[\rm(a)]
\item If $h$ is in $A(a,a)$, then $\Sigma(a)\amalg
  \Sigma^p(a)\simeq a\kos h$. 
\item If $\{ h_1,h_2,\ldots,h_t\}$ is in $A(a,a)$, then
  $\Sigma^t(a)$ is a direct summand of 
\[(\cdots ((a\kos h_1)\kos h_2)\cdots)\kos h_t.\]
\item $V(a\kos h) \subseteq V(a)$. 
\item The element $h^2$ is in
  $A(a\kos h,a\kos h)$. In particular,
 \[V(a\kos h)\subseteq V(\langle h\rangle)\cap V(a).\]
\end{enumerate}
\end{prop}
\begin{proof}
(a) This follows immediately from the triangle we constructed above. 

(b) Repeated use of (a) shows this. 

(c) We have the triangle $\Sigma^p(a)\to a\kos h\to
\Sigma(a)\to \Sigma^{p+1}(a)$ in $\A$. By Proposition
\ref{prop:elempropvariety} we infer that
$V(a\kos h)\subseteq V(\Sigma^p(a))\cup V(\Sigma(a))=V(a)$.

(d) From the triangle $a\extto{h\cdot 1_a} \Sigma^p(a)\to
a\kos h\to \Sigma(a)$ in $\A$ we get the exact sequence
\begin{multline}
\Hom_\A^*(b,a) \extto{\Hom_\A(b,(h\cdot 1_a))}
\Hom_\A^*(b,\Sigma^p(a))\to \notag \\
\Hom_\A^*(b,a\kos h)\to \Hom_\A^*(b,\Sigma(a))\notag
\end{multline}
for all objects $b$ in $\A$. A straightforward calculation shows that 
the map
\[\Hom_\A(b,(h\cdot 1_a))\colon \Hom_\A^*(b,a)\to
\Hom_\A^*(b,\Sigma^p(a))\] is given by multiplication by $h$ from the
left (up to sign). Since $\Hom_\A^*(b,\Sigma^p(a))=\Hom_\A^*(b,a)$, we
obtain the exact sequence
\[0\to \Hom_\A^*(b,a)/h\cdot\Hom_\A^*(b,a) \to 
\Hom_\A^*(b,a\kos h)\to \Ker(h\cdot -|_{\Hom_\A^*(b,a)})\to
0,\]
so that $h^2\cdot \Hom_\A^*(b,a\kos h)=(0)$ for all objects
$b$ in $\A$. It follows that  
$V(a\kos h)\subseteq V(\langle h\rangle)\cap V(a)$.
\end{proof}

If we impose the following extra condition on the action of $\C$ on $\A$, 
\begin{enumerate}
\item[(4)] The functor $-*a\colon \C\to \A$ is an exact functor for
all objects $a$ in $\A$. In this case the action is said to be
\emph{compatible} with the triangulation in $\C$,
\end{enumerate}
then we get an additional way of viewing $a\kos h$. Given the morphism
$h\colon \frake \to T^p(\frake)$ in $\C$, complete it to a triangle
\[\frake\extto{h} T^p(\frake) \to \frake\kos h\to T(\frake)\]
in $\C$. Here $\frake\kos h$ is unique up to a non-unique
isomorphism. For any object $a$ in $\A$ we get a commutative diagram
in $\A$, where the upper and the lower rows are triangles in $\A$.  
\[\xymatrix{%
\frake * a \ar[r]^-{h*1_a} \ar[dd]^{\frakl'} & 
T^p(\frake)*a \ar[r]\ar[d]^{\rho_p'} & 
(\frake\kos h)*a \ar[r]\ar[dd]^\wr & 
T(\frake)*a \ar[d]^{\rho'} \\
 & \Sigma^p(\frake *a) \ar[d]^{\Sigma^p(\frakl')}  &  & 
\Sigma(\frake*a) \ar[d]^{\Sigma(\frakl')} \\
a \ar[r]^-{h\cdot 1_a} & \Sigma^p(a) \ar[r] & a\kos h \ar[r] & \Sigma(a)}\] 
Since the two first vertical maps are isomorphisms, it follows that
$(\frake\kos h)*a\simeq a\kos h$. 

Computing the support variety of $a\kos h$ is the key to our
main result in this section. The proof is similar to the analogous
result in \cite[Proposition 4.3]{EHSST}. Next we show that the
inclusion in (d) above actually is an equality.
\begin{prop}\label{prop:intersecvar}
Let $h\colon \frake\to T^p(\frake)$ be in $H$. Then for any object $a$
in $\A$ 
\[V(a\kos h)=V(\langle h\rangle)\cap V(a).\]
\end{prop}
\begin{proof}\sloppy 
Choose a prime ideal $\frakp$ in $\Spec H$ lying
over $\langle h,A(a,a)\rangle$. Suppose that
$\cap_{b\in\A}A(b,a\kos h)$ is not contained in
$\frakp$. Then $(\Hom_\A^*(b,a\kos h))_{\frakp}=(0)$ for all
objects $b$ in $\A$. From the short exact sequence in the proof of the
previous result, we infer that
\[\Hom_\A^*(b,a)_{\frakp}=h\cdot \Hom_\A^*(b,a)_{\frakp}.\]
Since $\Hom_\A^*(b,a)_{\frakp}$ is a finitely generated
$H_{\frakp}$-module and $h$ is in $\frakp H_{\frakp}$, the Nakayama
Lemma implies that $\Hom_\A^*(b,a)_{\frakp}=(0)$. As $\Hom_\A^*(b,a)$
is a finitely generated $H$-module, the ideal $A(b,a)$ is not
contained in $\frakp$ for all objects $b$ in $\A$. In particular,
$A(a,a)$ is not contained in $\frakp$. This is a contradiction by the
choice of $\frakp$, hence
$\cap_{b\in\A}A(b,a\kos h)=A(a\kos h,
a\kos h)$ is contained in $\frakp$. It follows that $V(\langle
h\rangle)\cap V(a)\subseteq V(a\kos h)$. This completes the
proof using the previous result.
\end{proof}

Our main result of this section now follows directly from the above. 

\begin{thm}\label{thm:realization}
\begin{enumerate}[\rm(a)]
\item Let $a$ be an object in $\A$. Then any closed homogeneous
  subvariety of $V(a)$ occurs as the variety of some object in $\A$.
\item Suppose that $\A$ has a generator $\frakg$ in the sense that
  $\A=\Thick(\frakg)$. Then any closed homogeneous subvariety of $V(\frakg)$
  occurs as the variety of some object in $\A$. \qed
\end{enumerate}
\end{thm}

\section{Complexity and perfect and periodic objects}\label{section:4}
Throughout this section we keep the setup from the previous
section. Thus we fix a triangulated tensor category $\C$ acting on a
triangulated category $\A$ and a ring $H$ satisfying
Assumption~\ref{ass:standing}.

In this context we define the class of perfect objects as the objects
with support variety contained in $V(H^+)$ with
$H^+=\langle \sqrt{0_{H^0}}, H^{\geqslant 1}\rangle$, where
$\sqrt{0_{H^0}}$ is the nilradical of $H^0$.  We introduce a notion of
complexity of objects in $\A$, and we characterize the perfect objects
as those being of complexity $0$.  We also define and characterize
periodic objects in terms of complexity when $H^0$ is a local ring.

First we discuss the concept of complexity of objects in $\A$.
Condition (2) is equivalent to $R=H^0$ being a (commutative)
Noetherian ring and $H$ being a finitely generated
(graded-commutative) graded algebra over $R$ (as for commutative
graded rings). It follows from this that each graded part $H^i$ of $H$
is a finitely generated $R$-module.  Condition (3) says that any
$\Hom_\A^*(a,b)$ is a finitely generated $H$-module for all objects
$a$ and $b$ in $\A$, hence $\Hom_\A(a,\Sigma^i(b))$ is a finitely
generated $R$-module for all objects $a$ and $b$ in $\A$ and all $i$
in $\Z$. For a finitely generated $S$-module $M$, denote by
$\mingen_S(M)$ the minimal number of generators as an $S$-module. Then
we define the complexity of an object in $\A$ as follows.
\begin{defin}
The \emph{complexity} $\cx(a)$ of an object $a$ in $\A$ is given by
\[ \min\{ s\in\mathbb{N}_0\mid\forall\,
b\in \A, \exists\, r_b\in \mathbb{R};
\mingen_R \left(\Hom_\A(a,\Sigma^n(b))\right)\leqslant r_b|n|^{s-1},
\forall\, |n|\gg 0\},\]
if such $r_b$ and $s$ exist for all objects $b$ in $\A$. Otherwise we
set $\cx(a)=\infty$.  
\end{defin}
Note that since (3) $\Hom_\A^*(a,b)$ is a finitely generated
$H$-module for all objects $a$ and $b$ in $\A$ and (2) $H$ is
graded-commutative and Noetherian, we have that the complexity is
bounded by the polynomial growth of the graded parts of $H$ as
$R$-modules, which is finite. In addition, since $\Hom_\A^*(a,b)$ is a
finitely generated $\End_\A^*(a)$-module, the complexity of $a$ is
bounded, and therefore equal to
\[\min\{ s\in\mathbb{N}_0\mid\exists\, r\in \mathbb{R};
\mingen_R \left(\Hom_\A(a,\Sigma^n(a))\right)\leqslant r|n|^{s-1},
\forall\, |n|\gg 0\}.\] 

We collect some elementary properties of the complexity of objects
next, where we leave the proofs to the reader.
\begin{prop}
\pushQED{\qed}
\begin{enumerate}[\rm(a)]
\item $\cx(a)=\cx(\Sigma^i(a))$ for all objects $a$ in $\A$ and
  all integers $i$.
\item $\cx(a\amalg b)=\max\{\cx(a),\cx(b)\}$ for all objects $a$
  and $b$ in $\A$. 
\item If $a_1\to a_2\to a_3\to \Sigma(a_1)$ is a triangle in
  $\A$, then 
\[\cx(a_2)\leqslant \max\{\cx(a_1),\cx(a_3)\}.\qedhere\] 
\end{enumerate}
\end{prop}

One of the focal points in this section is the following notion of a perfect object.
\begin{defin}
An object $a$ in $\A$ is a \emph{perfect object} if $V(a)\subseteq
V(H^+)$.
\end{defin}
Now we characterize the perfect objects as those of complexity zero.
\begin{prop}\label{prop:trivvar}
Let $a$ be in $\A$. Then the following are equivalent. 
\begin{enumerate}[\rm(a)]
\item $a$ is perfect object in $\A$.
\item $\cx(a)=0$.
\item $\forall\, b\in \A, \exists\, n_b\in \mathbb{N}$ such that
  $\Hom_\A(a,\Sigma^i(b))=(0)$ for $|i|\geqslant n_b$.
\item $\forall\, b\in \A, \exists\, m_b\in \mathbb{N}$ such that
  $\Hom_\A(b,\Sigma^i(a))=(0)$ for $|i|\geqslant m_b$. 
\end{enumerate}
\end{prop}
\begin{proof} (b) implies (c): Assume that $\cx(a)=0$. This means that
  for all $b$ in $\A$ there exists $r_b$ in $\mathbb{R}$ such that
  $\mingen_R\left(\Hom_\A(a,\Sigma^n(b))\right) \leqslant r_b|n|^{-1}$
  for all $n$ with $|n|\gg 0$. This implies in turn that for all $b$
  in $\A$ there exists $n_b$ in $\mathbb{N}$ such that
  $\mingen_R\Hom_\A(a,\Sigma^n(b))=(0)$ for all $n$ such that
  $|n|\geqslant n_b$, which is the statement of (c).

  (c) implies (b): It follows immediately from the definition that
  $\cx(a)=0$.

(c) implies (d): Suppose that (c) holds. In particular,
$\Hom_\A(a,\Sigma^i(a))=(0)$ for all $i$ such that $|i|\geqslant n_a$ for
some integer $n_a$. Fix an object $b$ in $\A$. Since $\Hom_\A^*(b,a)$
is a finitely generated module over $\Hom_\A^*(a,a)$, say generated in
degrees $r_1<r_2<\cdots<r_t$ as a module over $\Hom_\A^*(a,a)$, then
we have that $\Hom_\A(b,\Sigma^i(a))=(0)$ for $i< r_1-n_a$ and for $i>
r_t+n_a$. The number $m_b=\max\{|r_1-n_a|, |r_t+n_1|\}+1$ depends only
on $b$ (when $a$ is fixed), and this number makes (d) hold true. 

(d) implies (a): Suppose that for all $b$ in $\A$ there exists $m_b$
in $\mathbb{N}$ such that $\Hom_\A(b,\Sigma^i(a))=(0)$ for all $i$
with $|i|\geqslant m_b$. In particular, $\Hom_\A(a,\Sigma^i(a))=(0)$ for
all $i$ with $|i|\geqslant m_a$. Then $H^{\geqslant m_a}$ is in $A(a,a)$. Since
$\sqrt{H^{\geqslant m_a}}=H^+$, it follows that $V(a)\subseteq V(H^+)$ and
$a$ is a perfect object.

(a) implies (c): Suppose that $a$ is a perfect object in $\A$. Then
$H^+\subseteq \sqrt{A(a,a)}$. Since $H$ is a finitely generated
algebra over $R$, we infer that $H^{\geqslant N}\subseteq A(a,a)$ for some
integer $N$ and $\Hom_\A^*(a,a)$  is a finitely generated $H/H^{\geqslant
  N}$-module. Hence $\Hom_\A(a,\Sigma^i(a))=(0)$ for all $i$ with
$|i|\geqslant n_a$ for some integer $n_a$. Since $\Hom_\A^*(a,b)$ is a
finitely generated module over $\Hom_\A^*(a,a)$ for all objects $b$ in
$\A$, it follows that there exists for all $b$ in $\A$ an integer
$n_b$ such that $\Hom_\A(a,\Sigma^i(b))=(0)$ for all $i$ with $|i|\geqslant
n_b$. Hence we have proved (c). This completes the proof of the
proposition. 
\end{proof}
\begin{remark}
If the triangulated category $\A$ has a generator $\frakg$ in the
sense that $\A=\Thick(\frakg)$, all of the above can be reformulated
in terms of $\frakg$ instead of for all objects $b$ in $\A$.
\end{remark}

We denote the full subcategory of $\A$ consisting of the perfect
objects by $\A^\perf$. The subcategory $\A^\perf$ is a thick
subcategory of $\A$, so we can form the Verdier quotient
$\A/\A^\perf$. Using this quotient we can define periodic objects as
follows.
\begin{defin}
  An object $a$ in $\A$ is \emph{periodic of period $n$} if $a$ is not
  perfect and $a\simeq \Sigma^n(a)$ in $\A/\A^\perf$ for some positive
  integer $n$, where $n$ is smallest possible.
\end{defin}
We have the following characterization of periodic objects.
\begin{prop}
  Assume that $R=H^0$ is a local ring.  Let $a$ be an object in
  $\A$. Then $a$ is a periodic object if and only if $\cx(a)=1$.
\end{prop}
\begin{proof}
Suppose that $a$ is a periodic object in $\A$, say $a\simeq
\Sigma^n(a)$ in $\A/\A^\perf$ for some non-zero integer $n$. This
means that there exist exact triangles in $\A$ of the form $a'\to a\to
p\to \Sigma(a')$ and $a'\to \Sigma^n(a)\to p'\to \Sigma(a')$ for some
perfect objects $p$ and $p'$. It follows that for $|i|\gg 0$ we have
that $\Hom_\A(a,\Sigma^i(a))\simeq \Hom_\A(a,\Sigma^{i+n}(a))$ and
hence that $\cx(a)\leqslant 1$. Since $a$ is not perfect, that is, $\cx(a)\geqslant
1$, we infer that $\cx(a)=1$.

Conversely, suppose that $\cx(a)=1$ and let $R=H^0$ be a local ring
with maximal ideal $\frakm$. Let $\overline{H}=H\otimes_R R/\frakm$,
which is a homomorphic image of $H$ and hence Noetherian. Let
$X=\Hom^*_\A(a,a)\otimes_R R/\frakm$, which is a finitely generated
$\overline{H}$-module and consequently Noetherian, since
$\Hom_\A^*(a,a)$ is a finitely generated $H$-module. Denote the
inclusion map $H^{+}\hookrightarrow H$ by $\nu$, and
$\overline{H^+}=\Im(\nu\otimes_R 1_{R/\frakm})=\overline{H}^{\geqslant
  1}\subset \overline{H}$. Consider the $\overline{H}$-submodule
\[(0:_X \overline{H^+})= \{ x\in X\mid \overline{H^+}x=(0)\}\] of $X$,
which is a finitely generated $\overline{H}$-module. Since
$\overline{H^+}\cdot (0:_X \overline{H^+})=(0)$, the module
$(0:_X \overline{H^+})$ is a finitely generated $R/\frakm$-module as
$\overline{H}/\overline{H^+}\simeq R/\frakm$. This implies that
$(0:_X \overline{H^+})$ only lives in a finite number of
degrees. Hence there exists an integer $w$ such that
$(0:_X \overline{H^+})_i=(0)$ for $i\geqslant w$.  Since $H$ is
positively graded and $\Hom_\A^*(a,a)$ is a finitely generated
$H$-module, $\Hom_\A(a,\Sigma^i(a))=(0)$ for $i \ll 0$. Therefore,
since $\cx(a)=1$, we infer that $\Hom_\A(a,\Sigma^i(a))\neq (0)$ for
infinitely many $i\gg 0$. This implies that $X_{\geqslant w}\neq (0)$,
and the set of associated primes $\Ass_{\overline{H}} X_{\geqslant w}$
is a finite set consisting of graded prime ideals. The union of these
primes is the set of homogeneous zero-divisors on $X_{\geqslant
  w}$. If $\overline{H^+}\subseteq \mathfrak{p}$ for some graded prime
in $\Ass_{\overline{H}} X_{\geqslant w}$, it follows that
$\overline{H^+}$ annihilates some non-zero element of
$X_{\geqslant w}$, which is a contradiction by the choice of $w$. We
conclude that $\overline{H^+}$ is not contained in any of the prime
ideals in $\Ass_{\overline{H}} X_{\geqslant w}$. The Prime Avoidance
Lemma implies that there exists a homogeneous
$X_{\geqslant w}$-regular element $h'$ in $\overline{H^+}$, that is,
$X_i\extto{h'} X_{i+|h'|}$ is an $(R/\frakm)$-monomorphism for
$i\geqslant w$. This gives rise to the following commutative diagram
\[\xymatrix{
& 
\Hom_\A^{\geqslant w}(a,a)\ar[d]\ar[r]^-h & 
\Hom_\A^{\geqslant w+|h|}(a,a)\ar[d]\\
0\ar[r] & 
\Hom_\A^{\geqslant w}(a,a)\otimes_R R/\frakm \ar[r]^-{h'}\ar[d] & 
\Hom_\A^{\geqslant w+|h'|}(a,a)\otimes_R R/\frakm\ar[d] \\
 & 0 & 0  }\]
where $h$ is an inverse image of $h'$ in $H$. Since $\cx(a)=1$,
there is a positive integer $r_a$ and epimorphisms of $R$-modules
$R^{r_a}\to \Hom_\A(a,\Sigma^n(a))$ for $n$ with $|n|\gg 0$. This
shows that the dimension of the $R/\frakm$-vectorspaces
$\Hom_\A^i(a,a)\otimes_R R/\frakm$ are bounded by $r_a$ for $|i|\gg 0$. Hence
the induced maps
\[\Hom_\A^i(a,a)\otimes_R R/\frakm\extto{h'} 
\Hom_\A^{i+|h'|}(a,a)\otimes_R R/\frakm\] are isomorphisms for
$|i|\gg 0$.  Since each graded piece $\Hom_\A^i(a,a)$ is a
finitely generated $R$-module, it follows from the commutative diagram
above that the map $\Hom_\A^i(a,a)\extto{h}
\Hom_\A^{i+|h|}(a,a)$ is an epimorphism for all $|i|\gg 0$. Let
$M_t= \Hom_\A^{\geqslant N+(t-1)|h|}(a,a)$ for integers $t$ and
$N$. Consider the map $M_t\extto{h\cdot -} M_{t+1}$. We have the
commutative diagram
\[\xymatrix{
0\ar[dr] & & & & \\
  & \Ker(h^2\cdot-)\ar[dr] & & & \\
0\ar[r] & \Ker(h\cdot -) \ar[r]\ar@{^(->}[u] 
        & M_t\ar[r]^{h\cdot-} \ar[dr]_{h^2\cdot-} 
        & M_{t+1}\ar[r]\ar[d]^{h\cdot-} 
        & 0\\
 & & & M_{t+2}\ar[d]\ar[dr] & \\
 & & & 0 & 0}\]
for some $N\gg 0$. Since $M_t$ is a finitely generated $H$-module for
all $t$ and $H$ is Noetherian, there exists an integer $L$ such that
$M_t\extto{h\cdot -} M_{t+1}$ is an isomorphism for all
$t\geqslant L$. Using the triangle
$a\extto{h\cdot -} \Sigma^{|h|}(a)\to a\kos h\to \Sigma(a)$ and the
long exact sequence induced from it, it follows that
$\Hom_\A^i(a,a\kos h)=(0)$ for $|i|\gg 0$.  Since $a\kos h$ is in the
thick subcategory generated by $a$, we infer that
$\Hom_\A^i(a\kos h,a\kos h)=(0)$ for $|i|\gg 0$ and $\cx(a\kos
h)=0$. It then follows that $a\kos h$ is a perfect object and that $a$
is a periodic object in $\A$.
\end{proof}

\section{Function objects}\label{section:5}

Let us begin with explaining one example of a function object. Let
$\C$ be the triangulated tensor category
$(\bfD^-(\L^e),-\otimes^{\mathbb{L}}_\L-,\L)$ for a finite dimensional
$k$-algebra $\L$ where $k$ is a field, and let $\A=\bfD^-(\mod\L)$.
Then $\C$ acts on $\A$, and we have that
\[\Hom_\A(B\otimes_\L^{\mathbb{L}}A,C) \simeq
\Hom_\A(A,\mathbb{R}\Hom_\L(B,C))\] for all objects $A$ and $C$ in
$\A$ and $B$ in $\C$. Then $\mathbb{R}\Hom_\L(-,-)\colon
\C^\op\times\A\to \A$ is called a left function object for the action
of $\C$ on $\A$. This section is devoted to recalling the definition
of and giving some elementary properties of such function objects. In
the next section we discuss the theory of support when the action of
$\C$ on $\A$ has a left function object.

Throughout this section let
$\C=(\C,\otimes,\frake,\fraka,\frakl,\frakr,T,\lambda,\rho)$ be a
triangulated tensor category acting on a small triangulated category
$\A=(\A,\Sigma)$. First we give the definition of a left function
object for an action of $\C$ on $\A$.
\begin{defin}
\begin{enumerate}[\rm(a)]
\item
A \emph{left function object} $F'$ for the action of $\C$ on $\A$
is a functor $F'\colon \C^\op\times \A\to \A$ such that
\begin{enumerate}[\rm(i)]
\item $F'(x,-)\colon \A\to \A$ is a covariant functor for each
  object $x$ in $\C$. 
\item $F'(-,a)\colon \C\to \A$ is a contravariant functor for
  each object $a$ in $\A$. 
\item there is an isomorphism
\[\Hom_\A(x*a,b)\to \Hom_\A(a,F'(x,b))\]
natural in all three variables. 
\end{enumerate}
\item A left function object $F'\colon \C^\op\times \A\to \A$
  is \emph{compatible}
\begin{enumerate}[\rm(i)]
\item \emph{with the triangulation of $\C$} if for each
  triangle $x\extto{u} y\extto{v} z\extto{w} T(x)$ in $\C$ and each
  object $a$ in $\A$, then 
\[F'(T(x),a)\extto{-F'(w,1)} F'(z,a)\extto{F'(v,1)} F'(y,a)\extto{F'(u,1)}
F'(x,a)\] 
is a triangle in $\A$. 
\item \emph{with the triangulation of $\A$} if for each
  triangle $a\extto{u} b\extto{v} c\extto{w} \Sigma(a)$ in $\A$ and
  each object $x$ in $\C$, then  
\[F'(x,a)\extto{F'(1,u)} F'(x,b)\extto{F'(1,v)} F'(x,c)\extto{F'(1,w)}
F'(x,\Sigma(a))\] 
is a triangle in $\A$. 
\end{enumerate}
\end{enumerate}
\end{defin}
In Section~\ref{section:3} we briefly discussed one occurrence of a
right function object in connection with the assumptions of the setup
for support varieties via the action of the graded endomorphism ring
$\End_\C^*(\frake)$. Next we give the precise definition of these
function objects.
\begin{defin}
\begin{enumerate}[\rm(a)]
\item
A \emph{right function object} $F''$ for the action of $\C$ on $\A$
is a functor $F''\colon \A^\op\times \A\to \C$ such that
\begin{enumerate}[\rm(i)]
\item $F''(a,-)\colon \A\to \C$ is a covariant functor for each
  object $a$ in $\A$. 
\item $F''(-,a)\colon \A^\op\to \C$ is a contravariant functor for
  each object $a$ in $\A$. 
\item there is an isomorphism
\[\Hom_\A(x*a,b)\to \Hom_\C(x,F''(a,b))\]
natural in all three variables. 
\end{enumerate}
\item A right function object $F''\colon \A^\op\times \A\to \C$
  is \emph{compatible with the triangulation of $\A$} 
\begin{enumerate}[\rm(i)]
\item if for each triangle
  $x\extto{u} y\extto{v} z\extto{w} \Sigma(x)$ in $\A$ and each object
  $a$ in $\A$, then
\[F''(T(x),a)\extto{-F''(w,1)} F''(z,a)\extto{F''(v,1)}
  F''(y,a)\extto{F''(u,1)} F''(x,a)\] 
is a triangle in $\C$. 
\item and if for each triangle $a\extto{u} b\extto{v} c\extto{w}
  \Sigma(a)$ in $\A$ and each object $x$ in $\A$, then
\[F''(x,a)\extto{F''(1,u)} F''(x,b)\extto{F''(1,v)}
  F''(x,c)\extto{F''(1,w)} F''(x,\Sigma(a))\] is a triangle in $\C$.
\end{enumerate}
\end{enumerate}
\end{defin}
If the action of $\C$ on $\A$ has a left or a right function object,
then each of them is unique up to isomorphism as stated next.
\begin{prop}
\begin{enumerate}[\rm(a)]
\item If the action of $\C$ on $\A$ has a left function object,
  then it is unique up to isomorphism.
\item If the action of $\C$ on $\A$ has a right function object,
  then it is unique up to isomorphism.\qed
\end{enumerate}
\end{prop}
The tensor category $\C$ acting on itself, may have both a left and a
right function object. They need not be isomorphic. However, the
occasions when they are isomorphic can be characterized as
follows. Recall that the tensor product $-\otimes-$ in $\C$ is
\emph{symmetric} if $x\otimes y\simeq y\otimes x$ via a natural
isomorphism in both $x$ and $y$.
\begin{prop}
Suppose that $\C$ acting on itself has a left and a right function
object.  Then the tensor product in $\C$ is symmetric if and only if
the left and the right function objects for $\C$ are isomorphic.\qed
\end{prop}
Next we point out how a left function object respects the triangulated
and the tensor structure in $\C$. The following proposition only deals
with left function objects, so we leave it to the reader to formulate
the corresponding results for right function objects.
\begin{prop}\label{prop:leftfuncobj}
Let $F'$ be a left function object for the action of $\C$ on $\A$. 
\begin{enumerate}[\rm(a)]
\item There is a natural isomorphism $F'(x,\Sigma(a))\simeq
  \Sigma(F'(x,a))$ for all objects $x$ in $\C$ and $a$ in $\A$. 
\item There is a natural isomorphism $F'(T(x),a)\simeq
  F'(x,\Sigma^{-1}(a))$ for all objects $x$ in $\C$ and $a$ in $\A$.  
\item There is a natural isomorphism $F'(y\otimes x,a)\simeq
  F'(x,F'(y,a))$ for all objects $x$ and $y$ in $\C$, and $a$ in
  $\A$. 
\item There is a natural isomorphism $a\simeq F'(\frake,a)$ for all
  objects $a$ in $\A$. \qed
\end{enumerate}
\end{prop}

\section{Support varieties for actions with a function object}\label{section:6}

This section is devoted to studying support varieties in a small
triangulated category $\A$ having a triangulated tensor category $\C$
acting on $\A$ with a left function object.  We indicate how this
restricts what we can expect to classify, and how we obtain some
control on the homomorphisms between one object and shifts of another
object in $\A$.

Throughout we keep the Assumption~\ref{ass:standing} and add for this
section the following. 
\begin{assumption}\label{ass:function}
For $\C$ and $\A$ the following holds:
\begin{enumerate}[(1)]\setcounter{enumi}{3}
\item There exists a left function object $F'\colon
  \C^\op\times\A\to \A$ for the action of $\C$ on $\A$.
\item The functor $-*a\colon \C\to \A$ is an exact functor for
  all objects $a$ in $\A$.
\end{enumerate}
\end{assumption}
In the presence of a left function object $F'$ for the action of $\C$
on $\A$ we denote the corresponding natural adjunction isomorphism by
\[\varphi=\varphi_{x,a,b}\colon \Hom_\A(x*a,b)\to \Hom_\A(a,F'(x,b))\] 
for all objects $a$ and $b$ in $\A$ and all $x$ in $\C$. 

\begin{remark}
(1)  It is tempting to believe that having this adjunction implies that
\[V(x * a, b) \subseteq V(a) \cap V(b).\]  However this is in general
not true as pointed out in \cite{BSS}.  But for $x = e\kos h$ we do
have $V(e\kos h * a, b) \subseteq V(a) \cap V(b)$.

(2)  When $\A$ is endowed with a theory of support varieties, the thick
subcategories $\X$ of $\A$ are sometimes given as
$\X_V=\{X\in\A\mid V(X)\subseteq V\}$ for some homogeneous subvariety
$V$ of $\Spec H$.  By the above remark one cannot expect that the
subcategories $\X_V$ are tensor subcategories of $\C$.  This is the
case for thick subcategories of the stable category of a $p$-group.
\end{remark}

Next we show that support varieties give us some control of the
homomorphisms between objects in $\A$. To this end we first need to
describe how a function object for the action acts on objects in $\A$.
\begin{lem}\label{lem:functionobjecttilde}
  Let $F'$ be a left function object for the action of $\C$ on $\A$,
  which is compatible with the triangulation in $\C$. Let $h\colon
  \frake\to T^p(\frake)$. Then $F'(\frake\kos h,a)$ is in $\Thick(a)$
  for all $a$ in $\A$.
\end{lem}
\begin{proof}
Let $h\colon \frake\to T^p(\frake)$, and let $a$ be in $\A$. Then
$F'(-,a)$ applied to the triangle $\frake\extto{h} T^p(\frake)\to
\frake\kos h\to T(\frake)$ in $\C$ gives rise to the triangle
\[F'(T(\frake),a)\to F'(\frake\kos h,a)\to F'(T^p(\frake),a)\to
F'(\frake,a)\] 
in $\A$. Since $F'(\frake,a)\simeq a$ and $F'(T^p(\frake),a)\simeq
\Sigma^{-p}(a)$ from Proposition~\ref{prop:leftfuncobj}, it follows
directly that $F'(\frake\kos h,a)$ is in $\Thick(a)$. 
\end{proof}
The following lemma is the last preliminary result we need before
proving the main results of this section.  
\begin{lem}\label{lem:functionobjectfactoring}
  Let $F'$ be a left function object for the action of $\C$ on $\A$.
  Fix an object $c$ in $\C$, and let $a$ and $b$ be two objects in
  $\A$. Let $\X$ be a thick subcategory of $\A$ such that
  $c*\X \subseteq \X$ and $F'(c,\X)\subseteq \X$. Then a morphism
  $f\colon c*a\to b$ factors through an object in $\X$ if and only if
  $\varphi(f)\colon a\to F'(c,b)$ factors through an object in $\X$.
\end{lem}
\begin{proof}
This follows directly from the fact that $\varphi$ is natural in all
three variables and from the assumptions. 
\end{proof}
Our first main result proves that any morphism
$f\colon a\to \Sigma^i(b)$ in $\A$ factors through an object $c$ with
variety contained in $V(a)\cap V(b)$ for any integer $i$.
\begin{prop}\label{prop:trivvarfactoring}
  Assume that the functor $-*x\colon \C\to \A$ is an exact functor for
  all objects $x$ in $\A$.  Let $F'$ be a left function object for the
  action of $\C$ on $\A$, and assume that it is compatible with the
  triangulation in $\C$.  For all $h\colon \frake\to T^p(\frake)$
    assume that the following diagram commutes for all $b$ in $\A$
\[\xymatrix{
F'(T^p(\frake), b) \ar[r]^-{F'(h,b)}\ar[d]^{\simeq} & F'(\frake,b)\ar@{=}[d] \\
F'(\frake,\Sigma^{-p}(b))\ar[d]^{\simeq} & F'(\frake,b)\ar[d]^{\simeq}\\
\Sigma^{-p}(b) \ar[r]^-{\Sigma^{-p}(h\cdot 1_b)} & b
},\]
where the vertical isomorphisms are given in \textup{Proposition
  \ref{prop:leftfuncobj}}. In particular, 
\[F'(e\kos h,b)\simeq \Sigma^{-p-1}(b\kos h).\]
Let $a$ and $b$ be two objects in $\A$. Then any morphism $a\to
\Sigma^i(b)$ factors through an object with support variety
contained in $V(a)\cap V(b)$, for any integer $i$.
\end{prop}
\begin{proof}
  Suppose that $A(a,a)=\langle h_1,h_2,\ldots,h_t\rangle$ for some
  homogeneous elements $h_1,h_2,\ldots,h_t$ in $H$. Then by
  Proposition~\ref{prop:decompannihilator} (b) the object
  $\Sigma^t(a)$ is a direct summand of
  $(\frake\kos h_1\otimes \cdots \otimes \frake\kos h_t)*a$, recalling
  that $\frake\kos h*a' \simeq a'\kos h$ for all homogeneous elements
  $h$ in $H$ and objects $a'$ in $\A$.  If a shift of a morphism
    $f\colon a\to \Sigma^i(b)$ factors through an object with variety in
    $V(a)\cap V(b)$, then also $f$ has the same property.  Therefore
    we can assume without loss of generality that $\Sigma^t(a)$ is
    $a$.  Consequently $\Hom_\A(a,\Sigma^i(b))$ is a direct summand
  of
  $\Hom_\A((\frake\kos h_1\otimes \cdots \otimes \frake\kos h_t)*a,
  \Sigma^i(b))$ for all integers $i$. Furthermore, note that
\[\Hom_\A((\frake\kos h_1\otimes\cdots\otimes \frake\kos h_t)*a, \Sigma^i(b)) \simeq 
\Hom_\A(a, F'(\frake\kos h_1\otimes \cdots\otimes \frake\kos h_t,
\Sigma^i(b))).\]
Hence by Lemma~\ref{lem:functionobjectfactoring}, we shall see that it
is sufficient to show that 
\[V(F'(\frake\kos h_1\otimes\cdots\otimes \frake\kos h_t,\Sigma^i(b)))
  \subseteq V(a)\cap V(b).\]  
Applying the last assumption multiple times we obtain that 
\[F'(\frake\kos h_1\otimes \cdots \otimes e\kos h_t, b) \simeq 
\Sigma^{-p_t-1}(\cdots (\Sigma^{-p_1-1}(b\kos h_1)\kos h_2)\cdots
)\kos h_t.\]
Then we infer that 
\[V(F'(\frake\kos h_1\otimes \cdots \otimes e\kos h_t, b)) = V(\langle
  h_1,\ldots,h_t\rangle) \cap V(b) = V(a)\cap V(b).\]
Let $\X=\{x\in \A\mid V(x)\subseteq V(a)\cap V(b)\}$.  Then with $c =
\frake\kos h_1\otimes \cdots \otimes \frake\kos h_t$, the
assumptions of Lemma~\ref{lem:functionobjectfactoring} are satisfied.
Hence we conclude that any morphism $f\colon a \to \Sigma^i(b)$ factors
through an object in $\X$.  This completes the proof.
\end{proof}

For the last result of this section we show that the support variety
of an indecomposable object in $\A$ is connected under additional
assumptions. The proof is similar to the one of \cite[Theorem
3.1]{Bergh}.
\begin{prop} 
We assume the following conditions:
\begin{enumerate}[\rm(i)]
\item The functor $-*a\colon \C\to \A$ is an exact functor for
  all objects $a$ in $\A$.
\item There exists a left function object $F'$ for the action of
  $\C$ on $\A$ compatible with the triangulation in $\C$. 
\item $\A$ is a Krull-Schmidt category and the idempotents split
  in $\A/\A^\perf$.
\item The degree zero part $H^0$ of $H$ is a local artinian ring.
\item For all $h\colon \frake\to T^p(\frake)$ the
  following diagram commutes for all $b$ in $\A$
\[\xymatrix{
F'(T^p(\frake), b) \ar[r]^-{F'(h,b)}\ar[d]^{\simeq} & F'(\frake,b)\ar@{=}[d] \\
F'(\frake,\Sigma^{-p}(b))\ar[d]^{\simeq} & F'(\frake,b)\ar[d]^{\simeq}\\
\Sigma^{-p}(b) \ar[r]^{\Sigma^{-p}(h\cdot 1_b)} & b
},\]
where the vertical isomorphisms are given in \textup{Proposition
  \ref{prop:leftfuncobj}}.
\end{enumerate}
Let $a$ be a non-perfect object in $\A$, such that $V(a)=V_1\cup V_2$
for some homogeneous varieties $V_1$ and $V_2$ with $V_1\cap
V_2\subseteq V(H^+)$.  Then, in the Verdier quotient $\A/\A^\perf$ the
object $a$ decomposes as $a_1\amalg a_2$ where $V(a_i)=V_i$ for
$i=1,2$.
\end{prop}
\begin{proof}
When $H^0$ is a local artinian ring, all perfect objects have the same
variety given by the ideal $\frakm_\gr=\langle \rad H^0, H^{\geqslant
  1}\rangle$, which is contained in the variety of any object in $\A$. 

Assume that $V(a)=V_1\cup V_2$ for some homogeneous varieties
$V_1=V(\mathfrak{a}_1)$ and $V_2=V(\mathfrak{a}_2)$ with
$V_1\cap V_2\subseteq V(H^+)$ and $\mathfrak{a}_i$ homogeneous ideals
in $H$ for $i=1,2$.  Denote by $\gamma(U)$ the polynomial growth of
the minimal number of generators of the graded pieces of a positively
graded $H$-module $U$ as a $H^0$-module.  The proof goes by induction
on $\gamma(H/\mathfrak{a}_1) + \gamma(H/\mathfrak{a}_2)$. Note that we
have $\cx(b)=\gamma(H/A(b,b))$ for any object $b$ in $\A$.
Furthermore, we have $\gamma(U/\rad(H^0) U) = \gamma(U)$, when
$U/\rad(H^0)U$ is considered as a graded module over $H/\rad(H^0)
H$. Hence we can think of $H$ as a graded algebra over the field
$H^0/\rad(H^0)$.

Assume first that one of $\gamma(H/\mathfrak{a}_i)$ is zero, say for
$i=1$.  Since $a$ is a non-perfect object, observe that
  $V(a) \neq \emptyset$.  Then by Theorem \ref{thm:realization} there
  exists a nonzero perfect object $p\in \A$. Then $p\amalg a \simeq a$
  in $\A/\A^\perf$.  As $V_1$ is the variety of a perfect object, we
  have $V(p) = V_1$.  Since $V(a) = V_1\cup V_2 = V_2$, we have that
  $V(a) = V_2$.  Hence we can choose $a_1=p$ and $a_2=a$.

Assume that $\gamma(H/\mathfrak{a}_i) > 0$ for $i=1,2$.  Since
$V_1\cap V_2\subseteq V(H^+)$, we infer that
$\gamma(H/\langle \mathfrak{a}_1, \mathfrak{a}_2\rangle) = 0$. Then we
can choose homogeneous elements $\eta_1\in \mathfrak{a}_1$ and
$\eta_2\in\mathfrak{a}_2$ of degrees $m$ and $n$, respectively, such
that
\[\gamma(H/\langle\mathfrak{a}_2,\eta_1\rangle)=\gamma(H/\mathfrak{a}_2) - 1\] 
and
\[\gamma(H/\langle\mathfrak{a}_1,\eta_2\rangle)=\gamma(H/\mathfrak{a}_1) - 1.\] 
We have that $V_i\subseteq V(\langle \eta_i\rangle)$ for $i=1,2$ and
$V(\langle \eta_1\eta_2\rangle)= V(\langle \eta_1\rangle)\cup V(\langle\eta_2\rangle)$,
which contains $V_1\cup V_2=V(a)$.  This implies that $\eta_1\eta_2$
is in $\sqrt{\Ann_H(\Hom_\A^*(a,a))}$. Choosing high enough powers of
$\eta_1$ and $\eta_2$, we can without loss of generality assume that
$\eta_1\eta_2$ is in $\Ann_H(\Hom_\A^*(a,a))$. By Proposition
\ref{prop:decompannihilator} (a) we have that
$\frake\kos\eta_1\eta_2*a\simeq \Sigma(a)\amalg \Sigma^{m+n}(a)$.

By the octahedral axiom there is a triangle $\theta$
\[\frake\kos \eta_2 \to \frake\kos \eta_1\eta_2\to
T^n(\frake\kos\eta_1)\to T(\frake\kos\eta_2)\]
in $\C$.  Then
\begin{align}
    V(\frake\kos\eta_2*a) & = V(\langle \eta_2\rangle)\cap V(a)\notag\\
    & = (V_1\cap V(\langle \eta_2\rangle))\cup V_2.\notag
\end{align}
By induction $\frake\kos\eta_2*a\simeq a_1\amalg a_2$ where
$V(a_1)=V_1\cap V(\langle \eta_2\rangle)$ and $V(a_2)=V_2$.  Similarly
we have
$V(T^n(\frake\kos\eta_1)*a)=V_1\cup (V_2\cap V(\langle
\eta_1\rangle))$,
so that $T^n(\frake\kos\eta_1)*a\simeq a_1'\amalg a_2'$ with
$V(a_1')=V_1$ and $V(a_2')=V_2\cap V(\langle \eta_2\rangle)$ .  Note
that both $V(a_1)\cap V(a_2')$ and $V(a_2)\cap V(a_1')$ are contained
in $V(H^+)$.  Then the triangle $\theta*a$ in $\A$ has the form
\[a_1\amalg a_2 \to \Sigma(a)\amalg \Sigma^{m+n}(a) \to a_1'\amalg
a_2' \extto{\psi} \Sigma(a_1)\amalg \Sigma(a_2).\] By the above
observations and Proposition~\ref{prop:trivvarfactoring} we have that
$\psi=\left(\begin{smallmatrix} u_1 & 0\\ 0 &
    u_2\end{smallmatrix}\right)$ in $\A/\A^\perf$.  The image of this
triangle in $\A/\A^\perf$ is again a triangle.  By the uniqueness
of the cone, it follows that we have two triangles in $\A/\A^\perf$  
\[a_i \to b_i \to a_i'\to \Sigma(a_i)\]
for $i=1,2$.  Hence we have an isomorphism $\varphi\colon
\Sigma(a)\amalg \Sigma^{m+n}(a)\to 
b_1\amalg b_2$ in $\A/\A^\perf$.  Using that these triangles can be
lifted back to $\A$, it is easy to see that $V(b_i) \subseteq V_i$ for
$i=1,2$.  Since $V(a) = V_1\cup V_2 = V(b_1\amalg
b_2) = V(b_1)\cup V(b_2)$, we infer that $V(b_i)=V_i$ for $i=1,2$.

Consider the natural compositions $f_1$ given by 
\begin{multline}
\Sigma(a)\to \Sigma(a)\amalg \Sigma^{m+n}(a) \extto{\varphi} b_1\amalg b_2
\extto{\left(\begin{smallmatrix} 1 & 0\\ 0 &
      0\end{smallmatrix}\right)} b_1\amalg b_2\notag\\ 
\extto{\varphi^{-1}} \Sigma(a)\amalg \Sigma^{m+n}(a) \to
\Sigma(a)
\end{multline} 
and $f_2$ given by 
\begin{multline}
  \Sigma(a)\to \Sigma(a)\amalg \Sigma^{m+n}(a) \extto{\varphi} b_1\amalg b_2
  \extto{\left(\begin{smallmatrix} 0 & 0\\ 0 &
        1\end{smallmatrix}\right)} b_1\amalg b_2\notag\\
  \extto{\varphi^{-1}} \Sigma(a)\amalg
  \Sigma^{m+n}(a) \to \Sigma(a).
\end{multline}
Using Proposition~\ref{prop:trivvarfactoring} we infer that
$f_2f_1=f_1f_2=0$ in $\A/\A^\perf$, so that $f_1$ and $f_2$ are
orthogonal idempotents and their sum is the identity on
$\Sigma(a)$. The claim follows from this.
\end{proof}

\section{Complete intersections}\label{section:7}

This section is devoted to reviewing our theory in the setting of
complete intersections.

Let $A = k\llbracket x_1,x_2,\ldots,x_n\rrbracket$ be the ring of formal power
series in $n$ indeterminants $\{x_1,x_2,\ldots, x_n\}$ over a field
$k$.  Let $(R,\frakm)$ be a complete intersection, where $R =
A/(a_1,\ldots,a_t)$ for a regular sequence $\{a_1,\ldots,a_t\}$ in the
square of the maximal ideal of $A$.  In \cite{Av,AvBu} support
varieties of finitely generated modules over $R$ (and more general
complete intersections) were defined in terms of $\Spec
R/\frakm[\chi_1,\ldots,\chi_t]$, where $\{\chi_1,\ldots,\chi_t\}$ is a
set of cohomological operators on $\Ext^*_R(M,M)$ of degree two for
any finitely generated $R$-module $M$. In our situation, the ring
$R[\chi_1,\ldots,\chi_t]$ can be viewed as a graded subring of the
Hochschild cohomology ring $\HH^*(R)$ (see \cite{SS}).

The derived category $\C'=\bfD^{b}(\mod R\otimes_k R)$ is a
triangulated tensor category via the derived tensor product
$-\otimes_R^{\mathbb{L}}-\colon \C'\times \C'\to \C'$, and it acts on
the derived category $\A=D^{b}(\mod R)$ via the derived tensor
product $-\otimes_R^{\mathbb{L}}-\colon \C'\times\A\to \A$. The stalk
complex $R$, with $R$ concentrated in degree zero, is the tensor
identity.  Then one can show that we have the following commutative
diagram
\[\xymatrix@C=40pt{
\C'\times \A \ar[r]^{-\otimes^{\mathbb{L}}_R-} & \A\\
\C\times \A \ar@{^(->}[u]\ar[r]^{-\otimes_R-} & \A
\ar@{^(->}[u]}
\]
where $\C=\Thick(R)$ inside $\C'$ and $-\otimes_R-$ represents the
total tensor product. Then $\C$ is a triangulated tensor category with
an action on $\A$. 

Consider $S=\End_\C^*(R)=\HH^*(R)$, which by our general theory is a
graded-commutative ring, where we note that $S^0=R$. The ring
$H=R[\chi_1,\ldots,\chi_t]$ can be viewed as a graded subring of $S$,
and the action on $\Ext^*_R(M,N)$ for two finitely generated
$R$-modules $M$ and $N$ factor through the inclusion into $S$. Since
$\Ext^*_R(M,N)$ is a finitely generated module over $H$ for all
finitely generated $R$-modules $M$ and $N$ (see \cite{G}), our
Assumptions~\ref{ass:standing} and \ref{ass:function} are satisfied
with left function object $F'=\mathbb{R}\Hom_R(-,-)$.  Hence we can
apply all the results obtained in the previous sections. In addition
we point out the following.

\begin{thm}\label{thm:ci}
  Let $(R,\frakm)$ be a complete intersection, where 
  \[R=k\llbracket x_1,x_2,\ldots,x_n\rrbracket/(a_1,a_2,\ldots,a_t)\] 
  for a field $k$ and for a regular sequence $\{a_1,a_2,\ldots,a_t\}$ in
  the square of the maximal ideal of the ring of formal power series
  $k\llbracket x_1,x_2,\ldots,x_n\rrbracket$ in $n$ indeterminants
  $\{x_1,x_2,\ldots, x_n\}$.
\begin{enumerate}[\rm(a)]
\item The perfect objects in $\bfD^{b}(\mod R)$ are the perfect 
  complexes. 
\item\sloppy Let $\frakb=(y_1,y_2,\ldots,y_t)$ be an ideal in $R$
  generated by elements in $\frakm$. Let $K(y_1,y_2,\ldots,y_t)$ be
  the Koszul complex on the set of the generators
  $\{y_1,y_2,\ldots,y_t\}$ of $\frakb$. 

  Then $K(y_1,y_2,\ldots,y_t)$ is a perfect complex and
  \[V(K(y_1,y_2,\ldots,y_t)) = V(\langle \frakb,H^{\geqslant 1} \rangle).\]
\item Let $X$ be a perfect complex, then 
\[\quad\quad\quad V(X)=\{\langle\frakp,H^+\rangle\in \Spec H\mid \frakp\in\Spec R
\text{\ with $X_\frakp\neq 0$ in $\bfD^{b}(\mod R_\frakp)$}\}.\]
\item View $M$ in $\mod R$ as a stalk complex
  concentrated in degree $0$, then 
\[V^*(M)=V(M)\cap V(\frakm H),\] 
where $V^*(M)$ is the variety of $M$ defined in \cite{AvBu}. 
\item For any $X$ in $\bfD^{b}(\mod R)$ we have that 
\[V(X) \subseteq \cup_{\frakp\in \Spec R} V(R/\frakp).\]
\end{enumerate}
\end{thm}
\begin{proof}
(a) The support variety of the stalk complex of $R$ is $V(H^+)$, so
that all perfect complexes over $R$ are perfect objects in $\bfD^{b}(\mod R)$. 

Conversely, let $X$ be a perfect object in $\bfD^{b}(\mod R)$. There exists a
complex of projective $R$-modules $p(X)$ and a quasi-isomorphism
$p(X)\to X$ such that $p(X)^i=(0)$ for $i\gg 0$. Also, there is an
integer $n$ such that $p(X)$ is exact to the left of $p(X)^{n-1}$. By
assumption
\[(0)=\Hom_{\bfD^{b}(\mod R)}(X,R/\frakm[i])\simeq
\Hom_{\bfK^{-,b}(\mod R)}(p(X),R/\frakm[i])\]  
for $|i|\geqslant m$. Let $N=\max\{m,n\}$. Then
$\Hom_{\bfK(R)}(p(X),R/\frakm[N+2])$ equals $\Ext^1_R(\Im
f^{N+1},R/\frakm)$. Hence $\Im f^{N+1}$ is projective, and by soft
truncation $p(X)$ is a perfect complex. 

(b) Let $\frakb=(y_1,y_2,\ldots,y_t)$ be an ideal in $R$ generated by
elements in $\frakm$. Consider the triangles $R\extto{y_i\cdot-} R\to
K_{y_i}\to R[1]$ for $i=1,2,\ldots,t$, where all $\{K_{y_i}\}$ are
perfect complexes. Then the Koszul complex $K(y_1,y_2,\ldots,y_t)$ is
given by $K_{y_1}\otimes_R \cdots\otimes_R K_{y_t}\otimes_R R$, which
clearly is a perfect complex. By Proposition~\ref{prop:intersecvar} we
infer that $V(K(y_1,\ldots,y_t))=V(\langle
y_1,y_2,\ldots,y_t\rangle)\cap V(R)= V(\langle\frakb, H^{\geqslant 1}\rangle)$.

(c) Let $X$ be a perfect complex. Then $\sqrt{\Ann_H(\Hom^*_{\bfD^{b}(\mod
    R)}(X,X))}$ contains $H^{\geqslant 1}$. Hence we only need to find
the variety of the ideal 
\[\fraka = \sqrt{\Ann_R(\Hom^*_{\bfD^{b}(\mod
    R)}(X,X))}=\sqrt{\cap_{i\in\mathbb{Z}}\Ann_R(\Hom_{\bfD^{b}(\mod R)}(X,X[i]))}\] 
in $\Spec R$. Since $\Ann_R(\Hom_{\bfD^{b}(\mod R)}(X,X)) \subseteq
\Ann_R(\Hom_{\bfD^{b}(\mod R)}(X,X[i]))$ for all integers $i$, we infer that 
\[\fraka = \sqrt{\Ann_R(\Hom_{\bfD^{b}(\mod R)}(X,X))}.\]
Then $\fraka\subseteq \frakp$ with $\frakp$ in $\Spec R$ if and only
if 
\[(0)\neq\Hom_{\bfD^{b}(\mod R)}(X,X)_{\frakp}\simeq \Hom_{\bfD^{b}(\mod
  R_{\frakp})}(X_{\frakp},X_{\frakp}).\] 
This is equivalent to that $X_{\frakp}\neq (0)$ in $\bfD^{b}(\mod
R_{\frakp})$. Hence the claim follows. 

(d) Let $M$ be in $\mod R$. Recall that 
\[V^*(M)=\Supp_{H/\frakm H}(\Ext^*_R(M,M)\otimes_R R/\frakm).\] 
Let $\frakp$ be in $V^*(M)$ and $\frakq=\varphi^*(\frakp)$. Then
$\frakp$ is in $V^*(M)$ if and only if
\begin{align}
(0) & \neq (\Ext^*_R(M,M)/\frakm \Ext^*_R(M,M))_{\frakp} \notag\\
    & \simeq \Ext^*_R(M,M)_{\frakq}/\frakm
    \Ext^*_R(M,M)_{\frakq}.\notag
\end{align}
Hence, $\frakp$ is in $V^*(M)$ if and only if
$(0)\neq\Ext^*_R(M,M)_{\frakq}\simeq \Hom_{\bfD^{b}(\mod R)}^*(M,M)_{\frakq}$ or
equivalently $\frakq$ is in $V(M)$. Since $\frakm H\subseteq
\varphi^*(\frakp)$ for all $\frakp$ in $\Spec (H/\frakm H)$, the claim
follows.

(e) Let $X$ be in $\bfD^{b}(\mod R)$. Then $X$ can be filtered in a finite
set of finitely generated $R$-modules $\{M_i\}_{i=1}^m$, and each such
$R$-module $M_i$ can be filtered in a finite set of $\{
R/\frakp_{i_j}\}_{j=1}^{m_i}$ with $\frakp_{i_j}$ in $\Spec R$. The
claim follows from this and Proposition~\ref{prop:elempropvariety}. 
\end{proof}
We can reformulate (c) and (d) of the previous result as follows. 
\begin{prop}
  Let $(R,\frakm)$ be a complete intersection, where 
  \[R=k\llbracket x_1,x_2,\ldots,x_n\rrbracket/(a_1,a_2,\ldots,a_t)\] 
  for a field $k$ and for a regular sequence $\{a_1,a_2,\ldots,a_t\}$ in
  the square of the maximal ideal of the ring of formal power series
  $k\llbracket x_1,x_2,\ldots,x_n\rrbracket$ in $n$ indeterminants
  $\{x_1,x_2,\ldots, x_n\}$. 

  Let $\varphi\colon H\to H/\frakm H$ and $\psi\colon H\to R$ be the
  natural ring homomorphisms, and consider the maps $\varphi^*\colon
  \Spec (H/\frakm)\to \Spec H$ and $\psi^*\colon \Spec R\to \Spec
  H$. 
\begin{enumerate}[\rm(a)]
\item If $X$ is a perfect complex in $\bfD^{b}(\mod R)$, then 
  \[V(X)=\psi^*(V_T(X)),\] 
  where $V_T(X)$ denotes the support of the perfect complex $X$ in the
  sense of \cite{T,N}.
\item If $M$ in $\mod R$ is viewed as a stalk complex
  concentrated in degree $0$, then 
\[\varphi^*(V(M)\cap V(\frakm H)) = V^*(M),\] 
where $V^*(M)$ is the variety of $M$ defined in
  \cite{AvBu}. 
\end{enumerate}
\end{prop}

\section{Group rings over commutative Noetherian
 local rings}\label{section:8} 
Here we apply our results to group rings of finite groups over
commutative Noetherian local rings.

For a finite group $G$ and a commutative Noetherian local ring $R$ we
can form the group ring $RG$. For a finitely generated $RG$-module $M$
it is known that $H^*(G,M)=\Ext^*_{RG}(R,M)$ is Noetherian as a module
over $H=H^*(G,R)=\Ext^*_{RG}(R,R)$ \cite{Ev,Go,V}. Let
$\C=(\bfD^{b}(\Proj RG), -\otimes_R-,\frake = R)$ act on
$\A=\bfD^{b}(\mod RG)$. Then our Assumptions~\ref{ass:standing} and
\ref{ass:function} are satisfied, as $-\otimes_R M\colon \C\to \A$ is
an exact functor for all modules/stalk complexes in $\A$ and
$F'=\Hom_R(-,-)\colon \C\times \A\to \A$ is a left function object for
the action which is compatible with the triangulation in $\C$. Hence
we can apply all of our results from the previous sections. In
addition we point out the following.

\begin{thm}
  Let $G$ be a finite group and let $R$ be a commutative Noetherian
  local ring.
\begin{enumerate}[\rm(a)]
\item The perfect objects in $\bfD^{b}(\mod RG)$ are the perfect 
  complexes. 
\item Let $X$ be a perfect complex in $\bfD^{b}(\mod RG)$, then 
\[\quad\quad\quad V(X)=\{\langle\frakp,H^{\geqslant 1}\rangle\in \Spec
H\mid \frakp\in\Spec R
\text{\ with $X_\frakp\neq 0$ in $\bfD^{b}(\mod R_\frakp G)$}\}.\]
\item If $R=k$ is an algebraically closed field and $M$ in $\mod
  RG$ is viewed as a stalk complex concentrated in degree $0$, then
\[V^*(M)=V(M)\] 
where $V^*(M)$ is the variety of $M$ defined in
\cite{Ca}. 
\end{enumerate}
\end{thm}
\begin{proof}
(a) As before we have that all perfect complexes in $\bfD^{b}(\mod
RG)$ are perfect objects in $\bfD^{b}(\mod RG)$. 

Conversely, let $X$ be a perfect object in $\bfD^{b}(\mod
RG)$. There exists a complex of projective $R$-modules $p(X)$ and a
quasi-isomorphism $p(X)\to X$ such that $p(X)^i=(0)$ for $i\gg
0$. Also, there is an integer $n$ such that $p(X)$ is exact to the
left of $p(X)^{n-1}$. We want to show that $p(X)$ can be softly
truncated to a perfect complex. 

To this end we use the following observations.  Given any finitely
generated $RG$-module $M$, there exists an $R$-epimorphism $f\colon
M\to R/\frakp$ for some $\frakp$ in $\Spec R$. Then the map
$\widetilde{f}\colon M\to R/\frakp G$ given by $\widetilde{f}(m) =
\sum_{g\in G} f(gm)g^{-1}$ is an $RG$-homomorphism, with non-zero
image not contained in $\frakm (R/\frakp G)$. Furthermore, if $M$
occurs as a kernel $0\to M\to P\to C\to 0$, where $P\to C$ is a
projective cover, then there exists an $RG$-homomorphism
$\varphi\colon M\to R/\frakp G$ which induces a non-zero element in
$\Ext^1_{RG}(C,R/\frakp G)$.

Using these observations the proof can be completed in a similar
fashion as the proof of Theorem~\ref{thm:ci} (a).

(b) The proof is similar to the proof of Theorem~\ref{thm:ci} (c).

(c) This is immediate as $\Hom_{\bfD^{b}(\mod RG)}^*(M,M)$ and
$\Ext^*_{RG}(M,M)$ are isomorphic as $H$-modules. 
\end{proof}

\section{Finite dimensional algebras}\label{section:9}

Throughout this section let $\L$ be an indecomposable finite
dimensional algebra over an algebraically closed field $k$ with
Jacobson radical $\rrad$.  We consider the action of the triangulated
tensor category $\C=\bfD^{b}(\B)$ on $\A=\bfD^{b}(\mod\L)$, where $\B$
is the full subcategory of $\L$-$\L$-bimodules which are projective as
a left and as a right module.  The section is devoted to studying
support varieties of objects in $\A$.

A theory of support varieties for $\mod\L$ was introduced in \cite{SS}
and further developed in \cite{EHSST} using the Hochschild cohomology
ring of $\L$.  To ensure the existence of a good theory of support,
two finiteness conditions \textbf{Fg1} and \textbf{Fg2} were
introduced in \cite{EHSST}, which in \cite[Proposition 5.7]{S} are 
shown to be equivalent to the \textbf{Fg} condition:
\begin{enumerate}[\rm(i)]
\item $\HH^*(\L)$ is a Noetherian algebra, 
\item $\Ext^*_\L(\L/\rrad, \L/\rrad)$ is a finitely generated
  $\HH^*(\L)$-module.
\end{enumerate}
We show that our Assumption~\ref{ass:standing} is equivalent to the
condition \textbf{Fg}.  Furthermore, a perfect object in
$\bfD^{b}(\mod\L)$ is characterized as a perfect complex, and the
support variety of $X$ in $\bfD^{b}(\mod\L)$ is shown to be contained
in the union of the support varieties of homology modules of $X$ and
in the union of the support varieties of stalk complexes of which $X$
is built up. Also, the support variety of a module, as defined in
\cite{SS}, coincides with the support variety when viewed as a complex
concentrated in one degree.
 
Our usual Assumptions~\ref{ass:standing} and \ref{ass:function}
translate into the following \textbf{facts} and requirements for the
setting of this section:
\begin{enumerate}[\bf(1)]
\item $\C=(\bfD^{b}(\B),-\otimes_\L-,\frake=\L)$ is a
  triangulated tensor category acting on the triangulated category
  $\A=\bfD^{b}(\mod\L)$.
\item[(2)] $H$ is a positively graded-commutative Noetherian ring with
  a homomorphism of graded rings $H\to \End_\C^*(\frake)=\HH^*(\L)$.
\item[(3)] The left $H$-module $\Hom_{\bfD^{b}(\mod\L)}^*(a,b)$ is
  finitely generated for all objects $a$ and $b$ in
  $\bfD^{b}(\mod\L)$.
\setcounter{enumi}{3}
\item The functor $-*a\colon \bfD^{b}(\B)\to
  \bfD^{b}(\mod\L)$ is an exact functor for all objects $a$ in
  $\bfD^{b}(\mod\L)$.
\item There exists a left function object
  \[F'=\mathbb{R}\Hom_\L(-,-)\colon
  \bfD^{b}(\B)\times\bfD^{b}(\mod\L)\to \bfD^{b}(\mod\L)\] for
  the action of $\bfD^{b}(\B)$ on $\bfD^{b}(\mod\L)$.
\end{enumerate}

We start by discussing the relationship between the condition (3) and
\textbf{Fg}. 
\begin{prop}
  Suppose $\L$ and $H$ satisfy condition \emph{(2)}. Then the
  following are equivalent.
\begin{enumerate}[\rm(i)]
\item $\Ext^*_\L(\L/\rrad,\L/\rrad)$ is a finitely generated
  $H$-module.
\item $\Ext^*_\L(M,N)$ is a finitely generated
  $H$-module for all $M$ and $N$ in $\mod\L$.
\item $\Hom_{\bfD^{b}(\mod\L)}^*(X,Y)$ is a finitely generated
  $H$-module for all $X$ and $Y$ in $\bfD^{b}(\mod\L)$. 
\end{enumerate}
\end{prop}
\begin{proof}
The fact that (i) and (ii) are equivalent is proven in
\cite[Proposition 1.4]{EHSST}. 

Since $\Ext_\L^*(M,N)\simeq \Hom_{\bfD^{b}(\mod\L)}^*(M,N)$ as
$H$-modules, we infer that (iii) implies (ii). Now assume that (ii) is
satisfied. Since any object in $\bfD^{b}(\mod\L)$ is isomorphic to a
bounded complex and $H$ is a Noetherian ring, it follows by induction
on the length of the finite complexes that (iii) holds.
\end{proof}

Hence, it follows that \textbf{Fg} is equivalent to the
Assumption~\ref{ass:standing}.  From now on we assume that these
conditions are satisfied for $\L$ and $H$ and in addition that $H^0$
is a (commutative) local artinian algebra.

As we noted in the above proof, $\Ext_\L^*(M,N)$ and
$\Hom_{\bfD^{b}(\mod\L)}^*(M,N)$ are isomorphic as $H$-modules for
all finitely generated $\L$-modules $M$ and $N$.  Hence we have the
following.
\begin{prop}
Any module $M$ in $\mod\L$ viewed as a stalk complex concentrated
in degree $0$ (any single degree) satisfies
\[V_H(M) = V(M),\]
where $V_H(M)$ denotes the support variety of $M$ defined in \cite{SS,
  EHSST}. 
\end{prop}

Since any object $X$ in $\bfD^{b}(\mod\L)$ is quasi-isomorphic to a
bounded complex, $X$ is in $\Thick(\L/\rrad)$. Using this, the
following is an easy consequence of the general theory and
\cite[Proposition 4.4]{SS}. 
\begin{prop}
For any complex $X$ in $\bfD^{b}(\mod\L)$ we have the following:
\begin{enumerate}[\rm(a)]
\item $V(X)=V(X,\L/\rrad)=V(X,X)=V(\L/\rrad,X)$.
\item $V(X) \subseteq V(\L/\rrad)=\Spec H$.\qed
\end{enumerate}
\end{prop}

Next we show that the variety of an object in $\bfD^{b}(\mod\L)$ is
contained in the union of the varieties of its homology modules and in
the union of the support varieties of the stalk complexes from which
it is built.  Furthermore the dimension of the support variety and the
complexity of an object are shown to be equal, and periodic stalk
complexes are characterized as eventually $\Omega$-periodic
$\L$-modules.
\begin{prop}\label{prop:elemproperties}
Let $X$ be in $\bfD^{b}(\mod\L)$. Then the following assertions
hold.
\begin{enumerate}[\rm(a)]
\item $V(X)\subseteq \cup_{i\in\mathbb{Z}}V(H^i(X))$.
\item $V(X)\subseteq \cup_{i\in\mathbb{Z}} V(X^i)$.
\item $\dim V(X) = \cx(X)$.
\item A module $M$ in $\mod\L$ is a periodic object in
  $\bfD^{b}(\mod\L)$ if and only if $M$ is an eventually $\Omega$-periodic
  $\L$-module.
\end{enumerate}
\end{prop}
\begin{proof}
  (a) Here again, we use that any object in $\bfD^{b}(\mod\L)$ is
  isomorphic to a bounded complex. If $X$ is a stalk complex, the
  claim clearly holds.  Suppose that the claim has been shown for all
  complexes of length $n-1$. Let $X\colon \cdots 0\to X^1\extto{d^1}
  X^2\extto{d^2} \cdots \extto{d^{n-1}} X^n\to 0\cdots$ be a complex
  of length $n$. Consider the triangle
\[\xymatrix@C=15pt{%
H^1(X)\ar[d] & & 
\ar@{..}[r]    &  
0\ar[r]\ar[d]  & 
\Ker d^1\ar[d]\ar[r] & 
0\ar[d]\ar@{..}[r] & & & \\
X\ar[d] & & 
\ar@{..}[r] & 
0\ar[r]\ar[d]  & 
X^1\ar@{=}[d]\ar[r] &
X^2\ar@{=}[d]\ar[r] &
\ar@{..}[r] & 
\ar[r] &
X^n\ar@{=}[d]\ar[r] & 
0\ar@{..}[r] & \\
Y &
\ar@{..}[r] & 
0 \ar[r] & 
\Ker d^1\ar[r]  & 
X^1\ar[r] &
X^2\ar[r] &
\ar@{..}[r] & 
\ar[r] &
X^n\ar[r] & 
0\ar@{..}[r] & }\]
By Proposition~\ref{prop:elementarypropvar} $V(X)\subseteq V(H^1(X))\cup
V(Y)$. It is easy to see that $Y$ is isomorphic to $Z\colon
\cdots 0\to X^2/\Im d^1\extto{\overline{d^2}} X^3\extto{d^3} \cdots
\extto{d^{n-1}} X^n\to 0\cdots$, so that $V(Y)=V(Z)$ and by
induction 
\begin{align}
V(X) & \subseteq V(H^1(X)) \cup (\cup_{i=2}^n
V(H^i(Z)))\notag \\
& = \cup_{i=1}^n V(H^i(X))\notag 
\end{align}
since $H^i(Z)=H^i(X)$ for $i=2,3,\ldots,n$. 

(b) Use similar arguments as in the proof of (a).

(c) By assumption $H^0$ is a local ring. Let $\mathfrak{r}$ be the
unique maximal ideal in $H^0$. Since $\mathfrak{r}$ and
$\mathfrak{r}H$ are nilpotent ideals and therefore contained in any
prime ideal, the dimensions
\begin{align}
\dim V(X) = & \dim H/\Ann_H\Hom_{\bfD^b(\mod\L)}^*(X,X)\notag\\
= & \dim \left((H/\mathfrak{r}H)/\Ann_{H/\mathfrak{r}H}
(\Hom_{\bfD^b(\mod\L)}^*(X,X)\otimes_{H^0} H/\mathfrak{r}H)\right)\notag
\end{align} coincide.
For this we observe that if $\overline{\eta}$ is in
\[\Ann_{H/\mathfrak{r}H} \Hom_{\bfD^b(\mod\L)}^*(X,X)\otimes_{H^0} H^0/\mathfrak{r},\]
the element $\eta$ is in \[\sqrt{\Ann_H\Hom_{\bfD^b(\mod\L)}^*(X,X)},\]
and if $\eta$ is in \[\Ann_H\Hom_{\bfD^b(\mod\L)}^*(X,X)\] then
$\overline{\eta}$ is in \[\Ann_{H/\mathfrak{r}H}
\Hom_{\bfD^b(\mod\L)}^*(X,X)\otimes_{H^0} H.\] Now the claim follows
from the fact that the Krull dimension of a finitely generated graded
algebra over a field equals the polynomial growth of that graded ring. 

(d) By the assumptions on $\L$, it must be a Gorenstein algebra
\cite{SS}. Then $\bfD^{b}(\mod\L)/\bfD^\perf(\mod\L)$ is equivalent
to $\uMCM(\L)$, where $\MCM(\L)$ is the full subcategory of the
$\mod\L$ consisting of the maximal Cohen-Macaulay modules and
$\uMCM(\L)$ is $\MCM(\L)$ modulo the ideal generated by the projective
modules \cite{Bu1986}. An equivalence 
\[F\colon \bfD^{b}(\mod\L)/\bfD^\perf(\mod\L)\to \uMCM(\L)\] 
is given as follows. Given an object $X$ in $\bfD^{b}(\mod\L)$,
construct a projective resolution $p(X)\to X$ of $X$. Then
$p(X)^{\geqslant t}$ is quasi-isomorphic to a maximal Cohen-Macaulay
module $M'$ for some $t$. Take a coresolution of $M'$ in projective
$\L$-modules and splice it with $p(X)^{\geqslant t}$ to get an acyclic
complex $X'$. Define $F(X)$ to be the image of the differential
starting in degree $0$ of $X'$. How $F$ acts on morphisms follows
naturally. Moreover, for any module $M$ we have that
$\Omega_\L^s(F(M))$ is isomorphic to $\Omega_\L^t(M)$ for some
positive integers $s$ and $t$.

Assume that $M\simeq M[n]$ in $\bfD^{b}(\mod\L)/\bfD^\perf(\mod\L)$
for some positive integer $n$. Then since $F$ is a triangle
equivalence, $F(M)\simeq F(M[n])\simeq \Omega_\L^{-n}(F(M))$, and
consequently $\Omega_\L^n(F(M))\simeq F(M)$. Then by the above remarks
it follows that $\Omega_\L^{n+p}(M)\simeq \Omega_\L^p(M)$ for some
positive integer $p$ and $M$ is eventually $\Omega$-periodic. 

Assume that $M$ is eventually $\Omega$-periodic. For any $\L$-module
$N$, it is easy to see that $N\simeq \Omega^t(N)[t]$ in
$\bfD^b(\mod\L)/\bfD^{\perf}(\mod\L)$ for all positive integers
$t$. Since by assumption $\Omega^{p+q}_\L(M)\simeq \Omega^q_\L(M)$ for
some positive integers $p$ and $q$, we infer that $M\simeq M[n]$ in
$\bfD^{b}(\mod\L)/\bfD^\perf(\mod\L)$ for some positive integer $n$,
that is, $M$ is periodic in $\bfD^{b}(\mod\L)$.
\end{proof}
\begin{remark}
  The inclusion in (a) is not an equality in general, as there exist
  perfect complexes with homology whose support varieties are not
  contained in $V(H^+)$.
\end{remark}

Now we show that the perfect objects in $\bfD^{b}(\mod\L)$ are
exactly the perfect complexes in $\bfD^{b}(\mod\L)$.
\begin{prop}
  Let $X$ be in $\bfD^{b}(\mod\L)$. Then an object $X$ is perfect if and
  only if $X$ is isomorphic to a perfect complex.
\end{prop}
\begin{proof}
  Suppose that $X$ is isomorphic to a perfect complex $P$. Since
  $V(\L)\subseteq V(H^+)$, it follows directly from
  Proposition~\ref{prop:elemproperties} (b) that $X$ is a perfect
  object.

Suppose that $X$ in $\bfD^{b}(\mod\L)$ is perfect.  Then
$\Hom_{\bfD^{b}(\L)}(X,\L/\rrad[i])=(0)$ for all $i$ such that
$|i|\geqslant N$ for some $N$. There exists a complex of projective
modules $pX$ and a quasi-isomorphism $pX\to X$ such that $(pX)^i=(0)$
for $i\gg 0$. By assumption
\[(0)=\Hom_{\bfD^{b}(\L)}(pX,\L/\rrad[i])\simeq
\Hom_{\bfK^{^-,b}(\L)}(pX,\L/\rrad[i])\] for $|i|\geqslant N$. In particular,
$\Hom_{\bfK^{-,b}(\L)}(pX,D(\L^\op)[i])=(0)$ when $|i|\geqslant N$, so that
$pX$ is exact beyond $N$ from degree zero. Then having
$\Hom_{\bfK^{-,b}(\L)}(pX,\L/\rrad[i])=(0)$ when $|i|\geqslant N$ implies that
$pX$ is split exact beyond degree $N$. Hence we can choose $pX$
such that $(pX)^i=(0)$ for $|i| > N$. This proves that $X$ is
isomorphic to the perfect complex $pX$.
\end{proof}

To discuss properties of complexes with complexity $d>0$ we need the
following characterization of the Koszul objects $\L\kos h$ for
$h\colon \L\to \L[n]$ in the setting of this section when $n\geqslant
1$.

\begin{lem}
  Let $h\colon \L\to \L[n]$ be in
  $\Hom_{\bfD^{b}(\B)}(\L,\L[n])$ for $n\geqslant 1$. Then there exists a
  $\L^e$-module $M_h$ such that $\L\kos h$ is quasi-isomorphic to a
  shift of the complex
\[\cdots 0\to M_h\to P^{-n+2}\to \cdots \to P^{-1}\to P^0\to
0\cdots\] 
with homology isomorphic to $\L$ in degrees $0$ and $n$.
\end{lem}
\begin{proof}
Fix a minimal projective resolution $\mathbb{P}\colon \cdots \to
P^{-n}\to P^{-n+1}\to \cdots \to P^{-1}\to P^0\to \L\to 0$ of $\L$
over $\L^e$. 

As an object in $\bfD^{b}(\B)$ the stalk complex $\L$ is
quasi-isomorphic to the truncated complex $\mathbb{P}_t\colon \cdots
\to P^{-n}\to P^{-n+1}\to \cdots \to P^{-1}\to P^0\to 0\cdots$. Then
$h$ is a map from $\mathbb{P}_t$ to $\L$ in degree $n$, and $\L\kos h$
is given by
\[\cdots\to P^{-n-1}\to P^{-n}
\extto{\left(\begin{smallmatrix} h\\ d^{-n}\end{smallmatrix}\right)} 
\L\amalg P^{-n+1}\extto{(0~d^{-n+1})} P^{-n+2}\cdots \to P^{-1}\to P^0\to
0\cdots.\]
We have the following commutative diagram
\[\xymatrix@C=30pt{%
\cdots\ar[r] & 
P^{-n}\ar[r]\ar@{=}[d] &
\Omega^n_{\L^e}(\L)\ar[r]\ar[d] & 
0 &
&
&
\\
\cdots\ar[r] & 
P^{-n}\ar[r]^-{\left(\begin{smallmatrix} h\\
      d^{-n}\end{smallmatrix}\right)} \ar[d] &
\L\amalg P^{-n+1}\ar[r]^-{(0~d^{-n+1})}\ar[d] &
P^{-n+2}\ar[r]\ar@{=}[d] & 
\cdots\ar[r] &
P^0\ar[r]\ar@{=}[d] &
0\\
\cdots \ar[r] & 
0\ar[r] &
M_h\ar[r] &
P^{-n+2}\ar[r] & 
\cdots\ar[r] &
P^0\ar[r] &
0
}\]
where the maps are coming from the pushout diagram
\[\xymatrix{
0\ar[r] &
\Omega^n_{\L^e}(\L) \ar[r]\ar[d]^h &
P^{-n+1}\ar[r]\ar[d] &
\Omega^{n-1}_{\L^e}(\L) \ar[r]\ar@{=}[d] &
0\\
0\ar[r] &
\L\ar[r] &
M_h\ar[r] &
\Omega^{n-1}_{\L^e}(\L) \ar[r] &
0}\]
Since the upper sequence is acyclic, $\L\kos h$ is quasi-isomorphic to
the complex in the lower row. It is clear from the above pushout
diagram that the homology in degrees $0$ and $n$ are isomorphic to
$\L$.
\end{proof}

Using this we show that $\L\kos h\otimes_\L X$ and $M_h\otimes_\L X$
have the same varieties.

\begin{lem}\label{lem:complexbimod}
\sloppy Let $X$ be in $\bfD^{b}(\mod\L)$, and let $h\colon \L\to \L[n]$ be in 
$\Hom_{\bfD^{b}(\B)}(\L,\L[n])$  for $n\geqslant 1$. Then 
\[V(\L\kos h\otimes_\L X)=V(M_h\otimes_\L X).\]
\end{lem}
\begin{proof}
Let $X\colon \cdots 0 \to X^0\extto{d^0} X^1\to \cdots \to X^t\to 
0\cdots$ be in $\bfD^{b}(\mod\L)$, and let $h\colon \L\to \L[n]$ be in 
$\Hom_{\bfD^{b}(\B)}(\L,\L[n])$.

Recall that $\L\kos h\colon \cdots 0\to M_h\extto{\beta_h} P^{-n+2}\to
\cdots \to P^{-1}\to P^0\to 0\cdots$. Let $P\colon \cdots 0 \to
P^{-n+1}\to \cdots \to P^0\to 0$.

Then 
\[M_h\otimes_\L X\extto{\beta_h\otimes\id_{X}} P\otimes_\L X
\to \L\kos h\otimes_\L X\to \]
is a triangle with $\cone(\beta_h\otimes\id_{X})\simeq
\L\kos h\otimes_\L X$. Since $P\otimes_\L X$ is a perfect
complex, it follows that $V(M_h\otimes_\L
X)=V(\L\kos h\otimes_\L X)$. 
\end{proof}
We end by showing how any object in $\bfD^{b}(\mod\L)$ can be
reduced to a perfect complex by tensoring with the special bimodules
$M_h$ introduced above.
\begin{prop}
Let $X$ be in $\bfD^{b}(\mod\L)$. Suppose that $\cx(X)=d$. 
\begin{enumerate}[\rm(a)]
\item There exist $d$ homogeneous elements $h_1$,
$h_2$,\ldots, $h_d$ in $H$ such that
$M_{h_1}\otimes_\L\cdots\otimes_\L M_{h_d}\otimes_\L X$ is
isomorphic to a perfect complex. 
\item If $\L$ is selfinjective, then as an object in
  $\bfK^{-,b}(\mod\L)$ the complex $M_{h_1}\otimes_\L\cdots\otimes_\L
  M_{h_d}\otimes_\L X$ is isomorphic to a direct sum of a perfect
  complex and an acyclic complex.
\end{enumerate}
\end{prop}
\begin{proof}
  Suppose that $\cx(X)=d$. Then there exist $d$ homogeneous
  elements $h_1$, $h_2$,\ldots, $h_d$ in $H^{\geqslant 1}$
  such that $\dim H/\langle
  h_1,h_2,\ldots,h_d,A(X,X)\rangle=0$. By Proposition~\ref{prop:intersecvar} the variety
  $V(\L\kos h_1\otimes_\L\cdots\otimes_\L \L\kos h_d\otimes_\L X)$ is
  equal to the variety of the ideal $\langle
  h_1,h_2,\ldots,h_d,A(X,X)\rangle$, hence contained in $V(H^+)$. Using
  Lemma~\ref{lem:complexbimod} the complex 
  $M_{h_1}\otimes_\L\cdots\otimes_\L M_{h_d}\otimes_\L X$ is
  perfect. The claims in (a) and (b) follow from this.
\end{proof}

\end{document}